\documentclass[11pt]{amsart}
\usepackage{amsbsy}
\usepackage{amsmath}
\usepackage{amsfonts}
\usepackage{graphicx,epsfig,subfigure,psfrag}

\usepackage{color}

\usepackage{changes}

\begin{document}

\newtheorem{theorem}{Theorem}[section]
\newtheorem{proposition}[theorem]{Proposition}
\newtheorem{lemma}[theorem]{Lemma}
\newtheorem{corollary}[theorem]{Corollary}
\newtheorem{definition}[theorem]{Definition}
\newtheorem{remark}[theorem]{Remark}

\newcommand{\tex}{\textstyle}
\numberwithin{equation}{section}
\newcommand{\ren}{\mathbb{R}^N}
\newcommand{\re}{\mathbb{R}}
\newcommand{\n}{\nabla}
\newcommand{\p}{\partial}
\newcommand{\iy}{\infty}
\newcommand{\pa}{\partial}
\newcommand{\fp}{\noindent}
\newcommand{\ms}{\medskip\vskip-.1cm}
\newcommand{\mpb}{\medskip}
\newcommand{\AAA}{{\bf A}}
\newcommand{\BB}{{\bf B}}
\newcommand{\CC}{{\bf C}}
\newcommand{\DD}{{\bf D}}
\newcommand{\EE}{{\bf E}}
\newcommand{\FF}{{\bf F}}
\newcommand{\GG}{{\bf G}}
\newcommand{\oo}{{\mathbf \omega}}
\newcommand{\Am}{{\bf A}_{2m}}
\newcommand{\CCC}{{\mathbf  C}}
\newcommand{\II}{{\mathrm{Im}}\,}
\newcommand{\RR}{{\mathrm{Re}}\,}
\newcommand{\eee}{{\mathrm  e}}
\newcommand{\LL}{L^2_\rho(\ren)}
\newcommand{\LLL}{L^2_{\rho^*}(\ren)}
\renewcommand{\a}{\alpha}
\renewcommand{\b}{\beta}
\newcommand{\g}{\gamma}
\newcommand{\G}{\Gamma}
\renewcommand{\d}{\delta}
\newcommand{\D}{\Delta}
\newcommand{\e}{\varepsilon}
\newcommand{\var}{\varphi}
\newcommand{\lll}{\l}
\renewcommand{\l}{\lambda}
\renewcommand{\o}{\omega}
\renewcommand{\O}{\Omega}
\newcommand{\s}{\sigma}
\renewcommand{\t}{\tau}
\renewcommand{\th}{\theta}
\newcommand{\z}{\zeta}
\newcommand{\wx}{\widetilde x}
\newcommand{\wt}{\widetilde t}
\newcommand{\noi}{\noindent}
\newcommand{\uu}{{\bf u}}
\newcommand{\xx}{{\bf x}}
\newcommand{\yy}{{\bf y}}
\newcommand{\zz}{{\bf z}}
\newcommand{\aaa}{{\bf a}}
\newcommand{\cc}{{\bf c}}
\newcommand{\jj}{{\bf j}}
\newcommand{\ggg}{{\bf g}}
\newcommand{\UU}{{\bf U}}
\newcommand{\YY}{{\bf Y}}
\newcommand{\HH}{{\bf H}}
\newcommand{\GGG}{{\bf G}}
\newcommand{\VV}{{\bf V}}
\newcommand{\ww}{{\bf w}}
\newcommand{\vv}{{\bf v}}
\newcommand{\hh}{{\bf h}}
\newcommand{\di}{{\rm div}\,}
\newcommand{\ii}{{\rm i}\,}
\def\I{{\mathbf{I}}}
\newcommand{\inA}{\quad \mbox{in} \quad \ren \times \re_+}
\newcommand{\inB}{\quad \mbox{in} \quad}
\newcommand{\inC}{\quad \mbox{in} \quad \re \times \re_+}
\newcommand{\inD}{\quad \mbox{in} \quad \re}
\newcommand{\forA}{\quad \mbox{for} \quad}
\newcommand{\whereA}{,\quad \mbox{where} \quad}
\newcommand{\asA}{\quad \mbox{as} \quad}
\newcommand{\andA}{\quad \mbox{and} \quad}
\newcommand{\withA}{,\quad \mbox{with} \quad}
\newcommand{\orA}{,\quad \mbox{or} \quad}
\newcommand{\atA}{\quad \mbox{at} \quad}
\newcommand{\onA}{\quad \mbox{on} \quad}
\newcommand{\ef}{\eqref}
\newcommand{\mc}{\mathcal}
\newcommand{\mf}{\mathfrak}

\newcommand{\ssk}{\smallskip}
\newcommand{\LongA}{\quad \rightarrow \quad}
\def\com#1{\fbox{\parbox{6in}{\texttt{#1}}}}
\def\N{{\mathbb N}}
\def\A{{\cal A}}
\newcommand{\de}{\,d}
\newcommand{\eps}{\varepsilon}
\newcommand{\be}{\begin{equation}}
\newcommand{\ee}{\end{equation}}
\newcommand{\spt}{{\mbox spt}}
\newcommand{\ind}{{\mbox ind}}
\newcommand{\supp}{{\mbox supp}}
\newcommand{\dip}{\displaystyle}
\newcommand{\prt}{\partial}
\renewcommand{\theequation}{\thesection.\arabic{equation}}
\renewcommand{\baselinestretch}{1.1}
\newcommand{\Dm}{(-\D)^m}


 \title[Effect of permeability in interface problems]{\bf The effect of the permeability in interface problems: impermeable barriers and free diffusion}

\author{Pablo \'Alvarez-Caudevilla, Cristina Br\"{a}ndle, M\'onica Molina-Becerra and Antonio Su\'arez}

\address{Departamento de Matem\'aticas, Universidad
Carlos III de Madrid, Legan\'es, Madrid, Spain} \email{pacaudev@math.uc3m.es, cbrandle@math.uc3m.es}
\address{Dpto. Matemática Aplicada II, Escuela Politécnica Superior and IMUS, Univ. de Sevilla, C/ Virgen de África, 7, 41011 Sevilla, Spain} \email{monica@us.es}
\address{Dpto. Ecuaciones Diferenciales y Análisis Numérico and IMUS, Univ. de Sevilla, Avda. Reina Mercedes s/n, 41012, Sevilla, Spain}
\email{suarez@us.es}


\thanks{This research was supported by the Ministry of Science, Innovation and Universities of Spain. PAC and CB were supported by grant PID2025-172146NB-I00, and MMB and AS were supported by grant PID2023-149509NB-I00. AS also has been supported by University of Sevilla under grants SOL2024-31708 and SOL2024-
31596, and  IMUS-María de Maeztu grant CEX2024-001517-M - Apoyo
a Unidades de Excelencia María de Maeztu, funded by MICIU/
AEI/ 10.13039/501100011033.
}

\date{\today}

\parskip5pt

\begin{abstract}
Focusing on reaction-diffusion models with an interface region dividing the domain, we study the stability of systems as the permeability parameter approaches zero or infinity under Kedem-Katchalsky boundary conditions imposed on the membrane or interface region. These boundary conditions establish that the flux across the interface manifold is proportional to the density jump across that region, with the constant of proportionality defined as the permeability parameter. We analyze both limiting behaviors: when permeability vanishes and when infinite flux is permitted. This analysis is carried out for both the well-known symmetric case, where permeability is identical in both directions across the membrane, and the asymmetric case, where it depends on the direction of the flux. For the asymmetric case, we obtain novel and surprising results that reveal a richer variety of behaviors in this class of interface problems and that differ from the scalar models.
\end{abstract}

\maketitle

\section{Introduction}

In many physical or biological systems modeled by differential equations, the domain of definition $\Omega\subset\mathbb{R}^N$, $N\geq 1$ may be divided by an interface or membrane $\Sigma\subset\Omega$, that splits $\Omega$ into two subsets $\Omega_1$ and $\Omega_2$ such that
\begin{equation}
\label{eq:domain}
\Omega=\Omega_1\cup \Omega_2 \cup \Sigma.
\end{equation}
Typically,
the solution for such a differential equation, or its derivatives, may have discontinuities across $\Sigma=\partial \Omega_1$, see
Figure \ref{fig:domain} where we have illustrated an example of $\Omega$ showing this configuration. This is the case, for instance, of reaction–diffusion models for protein cellular signaling, where all interactions
occur in a cell divided into two compartments, the nucleus and the cytoplasm, connected by the semipermeable
nuclear membrane and bounded by the impermeable cell membrane,~\cite{EliasClairambault2014}. The interaction of species throughout
the nuclear membrane is modeled by
the so-called Kedem–Katchalsky boundary conditions,~\cite{KedemKatchalsky1961}.  More precisely, if $u_1$ and $u_2$ denote the density of the two species, the export/import of the nuclear species and export/import of the cytoplasmic species is modeled by
$$
     \dfrac{\p u_1}{\p {\bf \nu}_1}= \gamma(u_2- u_1),\quad \dfrac{\p u_2}{\p {\bf \nu}_2}= \gamma( u_1- u_2), \qquad \text{on } \Sigma,
$$
where ${\bf \nu}_i$ is the unitary outward normal vector to $\Sigma$, which points outward with respect to $\Omega_i$, $i=1,2$,  and $\gamma>0$ is a real parameter, called the membrane permeability constant.
Therefore,
the smaller $\gamma$ is, the stronger the membrane is and, consequently, the lower the permeability is.

\begin{figure}[ht!]
    \centering
    \includegraphics[width=0.4\linewidth]{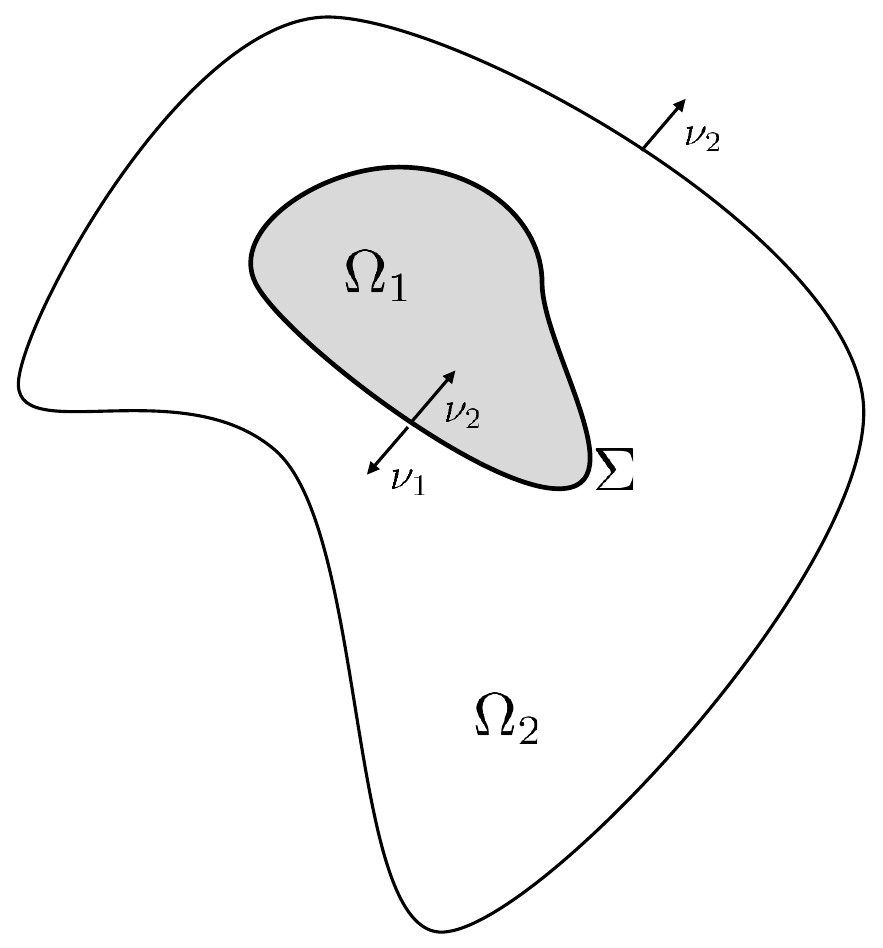}
    \caption{A possible configuration of the domain}
    \label{fig:domain}
\end{figure}

If the permeability of an interface depends on the direction of movement, that is, the interface is not symmetric,  the previous condition should be replaced by
$$
\frac{\partial u_1}{\partial \nu_1}=\gamma_1\left(u_2-u_1\right), \quad \frac{\partial u_2}{\partial \nu_2}=\gamma_2\left(u_1-u_2\right), \quad \text { on } \Sigma .
$$
in which $\gamma_i>0$ is the transmissivity of the membrane from $\Omega_i$ to the other subdomain.

 Interface and transmission problems of this type arise in a wide variety of models involving diffusion across internal barriers, including transport through biological membranes, cell and tissue compartmentalization, tumor invasion and spatial population dynamics. From the mathematical point of view, a growing literature has addressed coupled models through permeable interfaces, studying questions such as existence and uniqueness of solutions,
 spectral properties and the effect of asymmetric transmission conditions; see, for instance, \cite{CiavolellaPerthame2021,Ciavolella2022,Suarezetal,MaiaMolinaMoralesSuarez2024} and the references therein. More recently, interface transmission has also been considered in the context of intracellular compartmentalization and singularly perturbed cellular domains \cite{Bressloff2024}, while a general theory of renormalised solutions for reaction--diffusion systems coupled through nonlinear interface conditions has been developed in \cite{HopfTang2026}. Altogether, these works show that an internal membrane is not merely a perturbation of the exterior boundary conditions, but may substantially modify the spectral and qualitative properties of the underlying diffusion problem.
 In that sense, asymmetric conditions arise naturally in a variety of applied problems, for instance, individuals encountering the boundary between two different habitats may have different tendencies to move between them, as a consequence of habitat preference,  \cite{OvaskainenCornell2003}.

Furthermore, a fundamental step towards understanding the corresponding nonlinear models is to first characterize the spectral properties of associated linear problems. In particular, the behavior of their principal eigenvalue and eigenfunctions plays a key role in determining the qualitative properties of the solutions. Traditionally, those qualitative properties are obtained in terms of a parameter acting on the interior of the subdomains. In particular, if one assumes a logistic model such as $-\Delta u_i=\lambda u_i-au_i^2$ there are several works where the effect of the interface region is analyzed in terms of the growth rate $\lambda$; see, for instance, \cite{Suarezetal,WangSu} for further details. In fact, some interesting  behaviors, different from the traditional problems without interface regions, have also been obtained in \cite{PAC-Bran-Mol-Sua}
 showing that the existence of an interface region as the one we are considering here might change the way the solutions behave.
From those previous works, a natural question that arises here is the understanding of how the system behaves in terms of the permeability effect. In other words, changing the perspective and considering the permeability parameter as the main parameter of the system. In that sense, we would also like to answer such a question while assuming an asymmetric interface region. Thus, one of the aims of this paper is to provide answers to the understanding of the behavior of a system
when the permeability of the membrane is varied in only one direction, while the transmission properties in the opposite direction remain fixed.

Under these assumptions, and since the behavior of nonlinear problems is ascertained through the analysis of specific associated linear problems, given the domain configuration~\eqref{eq:domain} we focus  on the interface
eigenvalue problem given by
\begin{equation}
    \label{eq:main.system}
 \left\{\begin{array}{ll}
(-\Delta+c_1(x)) \varphi_{1} =\Lambda \varphi_1\quad   &\hbox{in } \Omega_1,\\
(-\Delta+c_2(x)) \varphi_{2} =\Lambda \varphi_2\quad   &\hbox{in } \Omega_2,
\end{array}\right.
\end{equation}
 where we assume that  $c_i\in { C^{0,\alpha}(\overline\Omega_i)}$ for some $0<\alpha<1$ and $\Lambda\in\mathbb{R}$. Together with the system, we consider boundary conditions given by an asymmetric Kedem-Katchalsky condition on the interface $\Sigma$ and a zero flux condition on the outside boundary of the domain. If we denote $\nu:=\nu_1=-\nu_2$ on $\Sigma$ de boundary conditions read:
\begin{equation}
    \label{eq:main.bc} \dfrac{\p u_1}{\p {\bf \nu}}= \gamma_1(u_2- u_1),\quad \dfrac{\p u_2}{\p {\bf \nu}}= \gamma_2( u_2- u_1), \qquad \text{on } \Sigma;
\qquad \partial_{\nu_2} \varphi_2=0, \qquad \text{on } \Gamma:=\partial\Omega.
\end{equation}

This article aims to establish how the principal eigenvalue of an elliptic interface problem depends on the permeability of the membrane, keeping in mind that this dependence determines the existence and behavior of positive solutions of the associated nonlinear problem. More precisely, we will provide a characterization of the two extreme regimes in which the membrane either prevents individuals from crossing in one direction, $\gamma_1\to0$, or becomes increasingly permeable in that direction, $\gamma_1\to\infty$, while $\gamma_2$ remains fixed. The special case in which the interface is symmetric, $\gamma:=\gamma_1=\gamma_2$ will also be covered.

The two limiting regimes have a clear interpretation in terms of the strength of the membrane, but, as we shall see, their effect strongly depends on whether permeability is symmetric or not. In the symmetric case, when $\gamma\to0$, the two subdomains become uncoupled, and the limiting principal eigenvalue selects the more favorable one {(Theorem~\ref{thm:symmetric.0})}. At the opposite extreme, when $\gamma\to+\infty$, the jump of the eigenfunction across the interface vanishes, and the membrane disappears in the limit, recovering the principal eigenvalue problem in the whole domain with the piecewise-defined potential, {see Theorem~\ref{thm:symmetric.infty}. These two limiting behaviors have also been recently obtained in~\cite{WangSu-preprint} in the framework of a general spectral theory for interface transmission problems. In fact, their results describe the convergence of the whole sequence of eigenvalues, being the mentioned symmetric principal eigenvalue limits  (Theorem~\ref{thm:symmetric.0} and Theorem~\ref{thm:symmetric.infty})  particular cases of their Theorems~2.4 and~2.5. However, the asymptotic behavior of the associated eigenfunctions is not explicitly characterized there.}  The asymmetric case is substantially different. If only one permeability coefficient tends to zero or infinity while the other one remains fixed, the limiting problem retains a directional coupling and, as a consequence, cannot in general be interpreted simply as either two isolated habitats or complete free diffusion across the whole domain.

The study of asymmetric interface problems, the core of this article, is therefore expected to be biologically richer and mathematically more involved than the symmetric case, since not only the solutions themselves, but also their gradients, may be discontinuous across the barrier. From an analytical point of view, the principal advantage of considering a symmetric permeability is that the problem becomes self-adjoint. In contrast, when the permeability is non-symmetric, some results no longer hold. However, one of the main results of this paper, Lemma \ref{lemma:variational}, is a variational characterization of the principal eigenvalue in a Hilbert space with a weight that depends on the relation $\gamma_1/\gamma_2$. Thanks to this transformation, we obtain a self-adjoint operator that allows us to characterize the eigenvalues variationally.

The variational characterization obtained in Lemma~\ref{lemma:variational} is then used to study the two singular limits of the permeability parameter. First, in Theorem~\ref{thm:gamma1.to.0} we consider the regime in which the permeability from $\Omega_1$ to $\Omega_2$ tends to zero, while the permeability in the opposite direction remains fixed. In this limit, the spectral problem separates into two different effective problems: a Neumann problem in $\Omega_1$ and a Robin-type problem in $\Omega_2$. The limiting principal eigenvalue is determined by the more favorable of these two configurations. The behavior of the corresponding eigenfunctions, however, reflects the asymmetry of the interface. If the limiting behavior is selected by the problem in $\Omega_2$, the component in $\Omega_1$ vanishes. If the problem in $\Omega_1$ is selected instead, the component in $\Omega_2$ generally remains non-trivial, because the transmission from $\Omega_2$ to $\Omega_1$ is still active. Thus, vanishing permeability in only one direction does not necessarily lead to a complete decoupling of the two regions.

The opposite regime is considered in Theorem~\ref{thm:gamma1.to.infty}, where the permeability from $\Omega_1$ to $\Omega_2$ becomes arbitrarily large while the reverse permeability remains fixed. In this case, the limiting spectral problem is again described by two different effective configurations: a Dirichlet problem in $\Omega_1$ and a Neumann problem in $\Omega_2$. Which one determines the limiting principal eigenvalue depends on their eigenvalues. The corresponding eigenfunctions also exhibit two distinct behaviors. In one case, the mass concentrates in $\Omega_1$, where a Dirichlet condition emerges at the interface. In the other, the component in $\Omega_2$ remains the dominant one, while the component in $\Omega_1$ is determined through the continuity imposed at the interface in the limit. This shows that increasing permeability in only one direction does not simply remove the membrane.

Therefore, the results stated in theorems~\ref{thm:gamma1.to.0} and~\ref{thm:gamma1.to.infty} contrast sharply with the symmetric case, in which both permeability coefficients vary simultaneously. When the common permeability tends to zero, the two subdomains become completely uncoupled, and the limiting behavior is determined by the more favorable of the two Neumann problems. On the other hand, when the common permeability tends to infinity, the jump across the interface vanishes, and the membrane disappears in the limit, yielding the standard eigenvalue problem in the whole domain with a piecewise-defined potential. This contrast between symmetric and directional permeability is the main qualitative phenomenon highlighted by our analysis.

\section{Preliminaries results}
\label{sect:preliminaries}
In this section, we collect some results related to the properties of the eigenvalues of an asymmetric interface system such as~\eqref{eq:main.system}--\eqref{eq:main.bc}. To this end, first, we fix some notations that will be used throughout the paper.

For convenience, we write $\mathbf{\Phi}=\left(\varphi_1, \varphi_2\right)$ with $\varphi_i$ defined in $\Omega_i$ and similarly $\mathbf{c}=\left(c_1, c_2\right)$, so that, if no confusion arises, we will write~\eqref{eq:main.system}--\eqref{eq:main.bc} as:
\begin{equation}
    \label{eq:main.short}
    \left\{\begin{array}{ll}
    -\Delta \mathbf{\Phi}+\mathbf{c}(x) \mathbf{\Phi}=\Lambda\mathbf{\Phi}(x) \quad &\text { in } \Omega,\\
     \mathcal{I}(\mathbf{\Phi})=0 \quad &\text { on } \Sigma \cup \Gamma.\end{array}\right.
\end{equation}

Finally, we will consider the following Banach spaces: denote as $\mathcal{L}^2$ the space $L^2(\Omega_1)\times L^2(\Omega_2)$; analogously, $\mathcal{H}^1$ will stand for $H^1(\Omega_1)\times H^1(\Omega_2)$. Other Lebesgue and Sobolev spaces will be defined by analogy.

As a general rule, the norm of a function $\mathbf{\Phi}$ is defined as the sum of the norms of $\varphi_i$ in the respective spaces.

\subsection{Scalar eigenvalue problems}

 Let $U\subset\mathbb{R}^N$ be a regular domain and
$$
\partial U=\Gamma_1 \cup \Gamma_2,
$$
with
$\Gamma_1 \cap \Gamma_2=\emptyset$ (some $\Gamma_i$ could be the empty set) and ${\bf n}$ is the outward unit normal vector field to $\partial U$.
Having in mind~\eqref{eq:main.short}, consider the uncoupled Robin problem
\begin{equation}
\label{eq:robin}
\begin{cases}-\Delta \varphi+c(x) \varphi=\sigma \varphi & \text { in } U, \\ \partial_{\bf n} \varphi+\gamma \varphi=0 & \text { on } \Gamma_1, \\ \partial_{\bf n} \varphi=0 & \text { on } \Gamma_2,\end{cases}
\end{equation}
with $c \in L^{\infty}(U)$ and $\gamma\in\mathbb{R}$, $\gamma\geq 0$. For that problem we denote
$$
\sigma_\gamma:=\sigma_1^U[-\Delta+c ; \mathcal{N}+\gamma, \mathcal{N}], \qquad\text{or}\qquad \sigma_\gamma:=\sigma_1^U[-\Delta+c ; \mathcal{N}+\gamma] \ \text{ if\ \  }\Gamma_2=\emptyset
$$
as the principal eigenvalue with  $\varphi_{\gamma}$ denoting the associated principal eigenfunction. Observe that the so-called homogeneous uncoupled interface problems
$$
\begin{cases}-\Delta \varphi_1+c_1(x) \varphi_1=\sigma \varphi_1 & \text { in } \Omega_1, \\ \partial_{\nu_1} \varphi_1+\gamma_1 \varphi=0 & \text { on } \Sigma,\end{cases}
\qquad\text{and}\qquad
\begin{cases}-\Delta \varphi_2+c_2(x) \varphi_2=\sigma \varphi_2 & \text { in } \Omega_2, \\ \partial_{\nu_2} \varphi_2+\gamma_2\varphi_2=0 & \text { on } \Sigma, \\ \partial_{\nu_2} \varphi=0 & \text { on } \Gamma,\end{cases}
$$
are special cases of~\eqref{eq:robin} defined in the respective domains $\Omega_i$. To simplify notation, in the sequel we denote their eigenvalues just as:
\begin{equation}
    \label{eq:uncoupled.eigenvalues}
    \sigma_1^{\Omega_1}[-\Delta+c_1 ; \mathcal{N}+\gamma_1],\qquad\text{and}\qquad
\sigma_1^{\Omega_2}[-\Delta+c_2; \mathcal{N}+\gamma_2].
\end{equation}

When $\Gamma_2=\emptyset$ let us also consider the eigenvalue problems of Dirichlet and Neumann type associated with $(-\Delta+c)$. That is
\begin{equation}
\label{eq:dirichlet.neumann}
\begin{cases}-\Delta \varphi+c(x) \varphi=\sigma \varphi & \text { in } U, \\
\varphi=0 & \text { on } \Gamma_1,
\end{cases}\quad\text{and}\quad
\begin{cases}-\Delta \varphi+c(x) \varphi=\sigma \varphi & \text { in } U, \\
\partial_{{\bf n}}\varphi=0 & \text { on } \Gamma_1,
\end{cases}
\end{equation}
with principal eigenvalues  denoted, respectively, by
$$
\sigma_{\mathcal{D}}:=\sigma_1^U[-\Delta+c ; \mathcal{D}] \quad \text{and}\quad \sigma_{\mathcal{N}}:=\sigma_1^U[-\Delta+c ;  \mathcal{N}]
$$
and principal eigenfunctions
$\varphi_{\mathcal{D}}$ and $\varphi_{\mathcal{N}}$.

We first recall a well-known property of the principal eigenvalue associated with Robin boundary conditions; see, e.g., \cite{Santi} for a general treatment of principal eigenvalues and their dependence on the boundary conditions. In the context of interface value problems, we refer also to~\cite{Suarezetal} for other properties of problem~\eqref{eq:robin},

\begin{lemma}
    \label{lemma:behaviour.scalar}
       Consider problem \eqref{eq:robin} with $\Gamma_2=\emptyset$. Then
    $$
    \sigma_\infty:=\lim_{\gamma \to \infty} \sigma_\gamma =\sigma_{\mathcal{D}},\quad\text{and}\quad  \sigma_0:=\lim_{\gamma \to 0^+}\sigma_\gamma=\sigma_{\mathcal{N}}.
   $$
    Moreover, if the associated positive eigenfunctions $\varphi_\gamma, \varphi_{\mathcal{D}}, \varphi_{\mathcal{N}}$ are normalized by
$\|\varphi\|_{L^2(U)}=1$, then
     $$
    \lim_{\gamma\to \infty} \varphi_\g=\varphi_{\mathcal{D}}
        ,\quad\text{and}\quad   \lim_{\gamma \to 0^+} \varphi_\g=\varphi_{\mathcal{N}}.
    $$
weakly in $H^1(U)$ and strongly in $L^2(U)$.
\end{lemma}

\begin{proof}
For $\gamma\geq 0$, the principal eigenvalue of~\eqref{eq:robin}, $\sigma_\gamma$, admits the variational characterization
$$
\sigma_\gamma
=
\inf_{\varphi\in H^1(U)\setminus\{0\}}
\frac{
\displaystyle
\int_U |\nabla \varphi|^2
+\int_U c(x)\varphi^2
+\gamma\int_{\partial U}\varphi^2
}{
\displaystyle
\int_U \varphi^2}.
$$
Equivalently,
$$
\sigma_\gamma
=
\inf_{\|\varphi\|_{L^2(U)}=1}
\left\{
\int_U |\nabla \varphi|^2
+\int_U c(x)\varphi^2
+\gamma\int_{\partial U}\varphi^2
\right\}.
$$
It follows immediately that $\sigma_\gamma$ is increasing and continuous with respect to $\gamma$. Moreover, by the monotonicity properties of principal eigenvalues (see, for example,~\cite{Santi}),
\begin{equation}
\label{eq:order.eigenvalues.scalar}
\sigma_{\mathcal N}\leq \sigma_\gamma\leq\sigma_{\mathcal D}.
\end{equation}
Hence the limits
$$
\sigma_0:=\lim_{\gamma\to0^+}\sigma_\gamma,
\qquad
\sigma_\infty:=\lim_{\gamma\to\infty}\sigma_\gamma
$$
exist and satisfy
$$
\sigma_{\mathcal N}\leq\sigma_0,\sigma_\infty
\leq\sigma_{\mathcal D}.
$$
Let ${\gamma_n}$ be a sequence converging either to $0^+$ or to $+\infty$, and let
$\varphi_n:=\varphi_{\gamma_n}$ be the positive principal
eigenfunction associated with  $\sigma_{n}$ and normalized by $\|\varphi_n\|_{L^2(U)}=1$. Then, the weak formulation of~\eqref{eq:robin} reads
\begin{equation}
\label{eq:weak_eig}
\int_U\nabla\varphi_n\cdot \nabla \psi
+\int_Uc(x)\varphi_n \psi
+\gamma_n\int_{\partial U}\varphi_n \psi
=
\sigma_{n}\int_U\varphi_n \psi,
\qquad \psi\in H^1(U).
\end{equation}
Taking $\psi=\varphi_n$,
from this weak formulation we conclude that
\begin{equation}
    \label{eq:proof.robin.bound}
\int_U|\nabla\varphi_n|^2
+\gamma_n\int_{\partial U}\varphi_n^2
\leq K,
\end{equation}
where $K$ is a positive constant independent of $n$. Hence
the family $\{\varphi_n\}$ is bounded in $H^1(U)$. Thus, up to a
subsequence,
$$
\varphi_n\rightharpoonup\varphi
\quad\text{weakly in }H^1(U),
\qquad
\varphi_n\rightarrow\varphi
\quad\text{strongly in }L^2(U).
$$
In particular,
$$
\|\varphi\|_{L^2(U)}=1.
$$

Suppose first that $\gamma_n\to0^+$.
 Since the trace operator
$H^1(U)\to L^2(\partial U)$ is continuous, the sequence of traces
${\varphi_n|_{\partial U}}$ is bounded in $L^2(\partial U)$. Hence, for
every $\psi\in H^1(U)$,
$$
\gamma_n\int_{\partial U}\varphi_n\psi\rightarrow0.
$$
Then, passing to the limit in~\eqref{eq:weak_eig}, we obtain
$$
\int_U\nabla\varphi\cdot \nabla\psi
+\int_Uc(x)\varphi\psi
=
\sigma_0\int_U\varphi\psi,
\qquad \psi\in H^1(U).
$$
Thus the limit function $\varphi$ is a normalized eigenfunction of the Neumann problem
associated with the eigenvalue $\sigma_0$ and
$\varphi\not\equiv0$. Therefore $\varphi$ is the principal Neumann
eigenfunction and
$$
\sigma_0=\sigma_{\mathcal N},
\qquad
\varphi=\varphi_{\mathcal N}.
$$
Since the principal eigenvalue is also simple in terms of its multiplicity, the limit is unique, and
hence the whole family converges.

Suppose now that $\gamma_n\to+\infty$. From~\eqref{eq:proof.robin.bound},
$$
\gamma_n\int_{\partial U}\varphi_n^2
\leq
K
$$
and consequently
$$
\int_{\partial U}\varphi_n^2
\leq
\frac{K}{\gamma_n}
\rightarrow0.
$$
On the other hand, the weak convergence in $H^1(U)$ implies weak
convergence of the traces. Therefore the trace of $\varphi$ vanishes on
$\partial U$, and hence $\varphi\in H_0^1(U)$.
Taking now $\psi\in H_0^1(U)$ in~\eqref{eq:weak_eig}, the boundary term
vanishes and
$$
\int_U\nabla\varphi_n\cdot \nabla\psi
+\int_Uc(x)\varphi_n\psi
=
\sigma_{n}\int_U\varphi_n\psi.
$$
Passing to the limit gives
$$
\int_U\nabla\varphi\cdot\nabla\psi
+\int_Uc(x)\varphi\psi
=
\sigma_\infty\int_U\varphi\psi,
\qquad \psi\in H_0^1(U).
$$
Since $\varphi\in H_0^1(U)$, $\|\varphi\|_{L^2(U)}=1$, and
$\varphi\geq0$, it follows that $\varphi$ is the normalized principal
Dirichlet eigenfunction. Consequently,
$$
\sigma_\infty=\sigma_{\mathcal D},
\qquad
\varphi=\varphi_{\mathcal D}.
$$
Again, by uniqueness of the normalized positive principal eigenfunction,
the whole family converges.
\end{proof}

\subsection{Interface eigenvalues}
Now, concerning the interface eigenvalue problem~\eqref{eq:main.short} we shall denote as $\Lambda$ and $\boldsymbol{\Phi}=(\varphi_1,\varphi_2)$, an eigenvalue and its corresponding associated eigenfunction of such a problem~\eqref{eq:main.short}. If $\Lambda \in \mathbb{R}$ and $\boldsymbol{\Phi} \in \mathcal{H}^1$, then \eqref{eq:main.short} holds in a weak sense; that is:
\begin{equation}
 \label{eq:weak.formulation}
\begin{aligned}
    \int_{\Omega_1} \nabla \varphi_1 \cdot\nabla \psi_1 &+ \int_{\Omega_2} \nabla \varphi_2\cdot \nabla \psi_2 +
    \int_{\Omega_1} c_1(x) \varphi_1 \psi_1+
    \int_{\Omega_2} c_2(x) \varphi_2 \psi_2 \\& +\int_{\Sigma} (\varphi_2-\varphi_1)(\gamma_2\psi_2-\gamma_1\psi_1)
    =\Lambda \left(\int_{\Omega_1}\varphi_1 \psi_1 +\int_{\Omega_2} \varphi_2 \psi_2\right)
\end{aligned}
\end{equation}
for all $\Psi=(\psi_1,\psi_2)\in\mathcal{H}^1$. Finally, $\Lambda$ is called a principal eigenvalue if $\boldsymbol{\Phi}>0$ in ${\Omega}$, with all the possible implications of having the eigenfunction $\boldsymbol{\Phi}$ denoted by the pair $(\varphi_1,\varphi_2)$. We will write that principal eigenvalue as
$$\Lambda_1(\gamma_1,\gamma_2):=\Lambda_1(-\Delta+c_1,-\Delta+c_2).$$

\begin{lemma}
\label{lemma:variational} Let $\Lambda_1(\gamma_1,\gamma_2)$ be the first eigenvalue of~\eqref{eq:main.short}. Then,
\begin{equation}
    \label{eq:Lambda.variational}
    \Lambda_1(\gamma_1,\gamma_2)=\inf_{{\bf \Phi}\neq 0}\frac{A({\bf \Phi},{\bf \Phi})}{b({\bf \Phi},{\bf \Phi})}
\end{equation}
for ${\bf \Phi}\in\mathcal{H}^1$, where
$$
A({\bf \Phi},{\bf \Phi})= \int_{\Omega_1} |\nabla \varphi_1|^2 +
    \int_{\Omega_1} c_1(x) \varphi_1^2 +\frac{\gamma_1}{\gamma_2}\left(\int_{\Omega_2} |\nabla \varphi_2|^2 +
  \int_{\Omega_2} c_2(x) \varphi_2^2 \right)
+\gamma_1\int_{\Sigma}(\varphi_2-\varphi_1)^2
$$
and
$$
b({\bf \Phi}, {\bf \Phi})=\int_{\Omega_1}\varphi_1^2  +\frac{\gamma_1}{\gamma_2}\int_{\Omega_2} \varphi_2^2.
$$
\end{lemma}

\begin{proof}
Let $(\varphi_1,\varphi_2)$ a principal eigenfunction of~\eqref{eq:weak.formulation} and $\Lambda_1(\gamma_1)$ the corresponding principal eigenvalue. Then  $(\varphi_1,\frac{\gamma_1}{\gamma_2}\varphi_2)$ verifies
$$
\begin{aligned}
    \int_{\Omega_1} \nabla \varphi_1\cdot \nabla \psi_1 &+
    \int_{\Omega_1} c_1(x) \varphi_1 \psi_1+\frac{\gamma_1}{\gamma_2}\left(\int_{\Omega_2} \nabla \varphi_2 \cdot \nabla \psi_2 +
  \int_{\Omega_2} c_2(x) \varphi_2 \psi_2 \right)\\& +\gamma_1\int_{\Sigma}(\varphi_2-\varphi_1)(\psi_2-\psi_1)
    =\Lambda(\gamma_1,\gamma_2) \left(\int_{\Omega_1}\varphi_1 \psi_1 +\frac{\gamma_1}{\gamma_2}\int_{\Omega_2} \varphi_2 \psi_2\right).
\end{aligned}
$$
for all $\Psi\in\mathcal{H}^1$.
Thus, this transformation has given rise to the weak formulation of a self-adjoint eigenvalue problem, being $\Lambda_1(\gamma_1, \gamma_2)$ the principal eigenvalue, and by standard arguments we can therefore conclude~\eqref{eq:Lambda.variational}.\end{proof}

\begin{remark}
    Observe that in the case of symmetric permeability $\gamma:=\gamma_1=\gamma_2$, the variational characterization~\eqref{eq:Lambda.variational} becomes just
\begin{equation}
    \label{eq:eigenvalue.gamma}
\Lambda_1(\gamma)
=
\inf_{{\bf \Phi}\neq 0}
\frac{
\displaystyle
\int_{\Omega_1} |\nabla \varphi_1|^2
+\int_{\Omega_1} c_1(x)\varphi_1^2+\int_{\Omega_2} |\nabla \varphi_2|^2
+\int_{\Omega_2} c_2(x)\varphi_2^2
+\gamma\int_{\Sigma}|\varphi_2-\varphi_1|^2
}{
\displaystyle
\int_{\Omega_1} \varphi_1^2+\int_{\Omega_2} \varphi_2^2
},
\end{equation}
 see~\cite{WangSu}.
\end{remark}

In the following lemmas, we will state some properties of the principal eigenvalue in terms of the permeability parameters that will be used to prove the main results of this paper. Without loss of generality, we state the results in terms of $\gamma_1$.

\begin{lemma}
\label{lemma:cont.monotone}
    The map $\gamma_1\mapsto \Lambda_1(\gamma_1,\gamma_2)$ is a continuous map.
    If additionally $\gamma_1=\gamma_2$ then the map $\gamma\mapsto \Lambda_1(\gamma)$ is non-decreasing.
\end{lemma}

\begin{proof}
Observe that for ${\bf \Phi}=(\varphi_1,\varphi_2)$ and $\gamma_2$ fixed, the Rayleigh quotient can be written as
    \begin{equation}
    \label{eq:Ray_gamma1}
    R(\gamma_1)=\frac{a_1+(a_2+s)\gamma_1}{b_1+b_2\gamma_1},
    \end{equation}
    where
$$
a_1=\int_{\Omega_1}|\nabla\varphi_1|^2+c_1\varphi_1^2,\quad a_2=\frac1{\gamma_2}\left(\int_{\Omega_2}|\nabla\varphi_2|^2+c_2\varphi_2^2\right),\quad s=\int_\Sigma(\varphi_2-\varphi_1)^2,
$$
and
$$
b_1=\int_{\Omega_1}\varphi_1^2,\qquad b_2=\frac1{\gamma_2}\int_{\Omega_2}\varphi_2^2.
$$
Consequently, from the definition of the function \eqref{eq:Ray_gamma1} and such that $\Lambda_1(\gamma_1,\gamma_2)$ is given by~\eqref{eq:Lambda.variational} we can conclude the continuity of the map $\gamma_1\mapsto \Lambda_1(\gamma_1,\gamma_2)$.

The monotonicity when $\gamma_1=\gamma_2$  is a direct consequence of the variational definition of the first eigenvalue~\eqref{eq:eigenvalue.gamma}.
\end{proof}

\begin{remark}
\label{rem:rayleigh}
    When $\gamma_1\neq \gamma_2$, then the monotonicity is, in general, not true. Indeed, using \eqref{eq:Ray_gamma1} it follows that
$$
\frac{d}{d\gamma_1}R(\gamma_1)=\frac{(a_2+s)b_1-a_1b_2}{(b_1+\gamma_1b_2)^2},
$$
which can be positive or negative.
\end{remark}

Finally, we provide a bound for the first eigenvalue \eqref{eq:Lambda.variational} in terms of the eigenvalues of different uncoupled problems.
 \begin{lemma}[\cite{Suarezetal}]
\label{autovalores_acotados}
Let $\Lambda_1(\gamma_1,\gamma_2)$ be the first eigenvalue of~\eqref{eq:main.short}. Then
$$
\Lambda_1(\gamma_1,\gamma_2)<\min \left\{\sigma_1^{\Omega_1}\left[-\Delta+c_1 ; \mathcal{N}+\gamma_1\right], \sigma_1^{\Omega_2}\left[-\Delta+c_2 ; \mathcal{N}+\gamma_2\right]\right\},
$$ where $\sigma_1^{\Omega_1}\left[-\Delta+c_1 ; \mathcal{N}+\gamma_1\right]$ and $\sigma_1^{\Omega_2}\left[-\Delta+c_2 ; \mathcal{N}+\gamma_2\right]$ are the principal eigenvalues of the uncoupled problems given in~\eqref{eq:uncoupled.eigenvalues}.
\end{lemma}

\begin{proof} For the sake of completeness, we include the proof here, which can also be read in~\cite{Suarezetal}.
Consider ${\bf \Phi}=\left(\varphi_1, \varphi_2\right)$ a positive eigenfunction of~\eqref{eq:main.short} associated with  $\Lambda_1(\gamma_1,\gamma_2)$. Then,
$$
-\Delta \varphi_1+\left(c_1-\Lambda_1(\gamma_1, \gamma_2)\right) \varphi_1=0 \quad \text { in } \Omega_1, \text{\qquad and\qquad} \partial_\nu \varphi_1+\gamma_1 \varphi_1=\gamma_1 \varphi_2>0 \quad \text { on } \Sigma.
$$
Hence, $\varphi_1$ is a positive (strict) supersolution of
$$
\left\{
\begin{array}
    {ll}
    -\Delta \varphi_1+\left(c_1-\Lambda_1(\gamma_1, \gamma_2)\right) \varphi_1=0,\qquad& \text { in } \Omega_1,\\
    \varphi_1+\gamma_1 \varphi_1=0,& \text { on } \Sigma.
\end{array}\right.
$$
As a consequence of \cite[Proposition 3.1 3]{Suarezetal} we obtain that
$$\sigma_1^{\Omega_1}\left[-\Delta+c_1-\Lambda_1(\gamma_1, \gamma_2) ; \mathcal{N}+\gamma_1\right]>0,
$$
or equivalently, $\sigma_1^{\Omega_1}\left[-\Delta+c_1 ; \mathcal{N}+\gamma_1\right]>\Lambda_1(\gamma_1, \gamma_2)$.

A similar argument shows that $
\Lambda_1(\gamma_1,\gamma_2 )< \sigma_1^{\Omega_2}\left[-\Delta+c_2 ; \mathcal{N}+\gamma_2\right]
$.
\end{proof}

\begin{corollary}
\label{corol:Lambda.less.D}
Let $\Lambda_1(\gamma_1,\gamma_2)$ be the first eigenvalue of~\eqref{eq:main.short}. Then
 \begin{equation}
    \label{eq:bound.eigenvalue}
    \Lambda(\gamma_1, \gamma_2)< \min\left\{
    \sigma_1^{\Omega_1}[-\Delta+c_1;\mathcal{D}]; \sigma_1^{\Omega_2}[-\Delta+c_2;\mathcal{D}]
    \right\}
\end{equation}
\end{corollary}

\begin{proof}
    This is a direct consequence of the monotonicity properties of the eigenvalues with respect to the domain, see~\eqref{eq:order.eigenvalues.scalar} and~\cite{Santi}.
\end{proof}

The following lemma, where we give a sharper bound for the principal eigenvalue, will be very important in obtaining the main results of the paper later on.
\begin{lemma}
Let $\Lambda_1(\gamma_1,\gamma_2)$ be the first eigenvalue of~\eqref{eq:main.short}. Then
\begin{equation}
    \label{eq:cota.Lambda.autovalor}
    \Lambda_1(\gamma_1, \gamma_2)\leq\frac{a_1+\frac{\gamma_1}{\gamma_2}\sigma_1^{\Omega_2}[-\Delta +c_2,\mathcal{N}]}{b_1+\frac{\gamma_1}{\gamma_2}}
\end{equation}
{where
$$
a_1=\int_{\Omega_1}|\nabla\varphi_1|^2+c_1(x)\varphi_1^2,
\quad
b_1=\int_{\Omega_1}\varphi_1^2
$$
being $\varphi_1$ a regular function such that $\varphi_1= \varphi_2$ on $\Sigma$, with
$\varphi_2$ is the positive eigenfunction associated with the first eigenvalue $\s_1^{\Omega_2}[-\Delta +c_2,\mathcal{N}]$ of the Neumann problem \eqref{eq:dirichlet.neumann}
  with $\|\varphi_2\|_{L^2(\Omega_2)}=1$.}
 \end{lemma}

\begin{proof}
  {
Taking $\Phi=(\varphi_1,\varphi_2)$ in Lemma \ref{lemma:variational} and by definition \eqref{eq:Lambda.variational}, we get  that
    $$
    \Lambda_1(\gamma_1, \gamma_2)\leq \frac{a_1+\dfrac{\gamma_1}{\gamma_2}\left(\displaystyle\int_{\Omega_2}|\nabla\varphi_2|^2+c_2(x)\varphi_2^2\right)+\gamma_1\int_\Sigma(\varphi_2-\varphi_1)^2}{b_1+\dfrac{\gamma_1}{\gamma_2}\displaystyle\int_{\Omega_2}\varphi_2^2}.
$$
Now, observe that
$$
\int_\Sigma(\varphi_2-\varphi_1)^2=0
$$
and
$$
\int_{\Omega_2}|\nabla\varphi_2|^2+c_2(x)\varphi_2^2= \s_1^{\Omega_2}[-\Delta +c_2,\mathcal{N}],
$$
from which the result follows.
}
\end{proof}

\section{Limit behavior}
In this section we study the effect of the permeability on the behavior of the species.
In the sequel, we will write
$\Lambda(\gamma_1)$ and ${\bf \Phi}_{\gamma_1}$ to denote the first eigenvalue of the asymmetric problem and its corresponding eigenfunction, for a fixed value of the parameter $\gamma_2$, and $\Lambda(\gamma)$ and ${\bf \Phi}_{\gamma}$ for the symmetric problem, when $\gamma=\gamma_1=\gamma_2$.

 First, let us consider for the rest of this section a sequence $\gamma_{1,n}$ converging to $0$ or $\infty$ (depending on the behavior of $\gamma_1$) as $n\to\infty$, and let $\Lambda_n:=\Lambda(\gamma_{1,n})$ and ${\bf \Phi}_n=(\varphi_{1,n},\varphi_{2,n})\in\mathcal{H}^1$ be the principal eigenvalue and the eigenfunction of~\eqref{eq:main.short}, with ${\bf \Phi}_n$
normalized so that $\|{\bf \Phi}_n\|_{\mathcal{H}^1}=1$, that is
\begin{equation}
    \label{eq:normalization}
\|\varphi_{1,n}\|_{H^1(\Omega_1)}+\|\varphi_{2,n}\|_{H^1(\Omega_2)}=1.
\end{equation}

\begin{lemma}
\label{lemma:Lambda.convergence}
Let $\{\gamma_{1,n}\}$ be a sequence such that
$\gamma_{1,n}\to 0$ (respectively $\gamma_{1,n}\to+\infty$), and let $\Lambda_n$ denote the corresponding
principal eigenvalue of \eqref{eq:main.short}.  Then, up
to a subsequence, there exists $\Lambda_0\in\mathbb{R}$ (respectively $\Lambda_\infty\in\mathbb{R}$) such that
$$
\Lambda_n\rightarrow\Lambda_0, \quad \text{(respectively $\Lambda_n\rightarrow\Lambda_\infty$)}
\qquad\text{as }n\to\infty.
$$
\end{lemma}

\begin{proof}
The proof is straightforward just noting that  $\{\Lambda_n\}$ is bounded, see Lemma~\ref{autovalores_acotados}.
\end{proof}

\subsection{An impermeable interface: the limit behavior as $\gamma_1\to 0$}
Let us start analyzing the case when the permeability effect, in terms of the parameter $\gamma_1$, is canceled out, once $\gamma_2$ is fixed. Such an effect shows a very interesting and new result while assuming the asymmetric case. In particular, when one of the permeability parameters, in this case $\gamma_1$, goes to zero, the eigenfunction for that side of the domain, i.e., $\varphi_1$, vanishes as $\gamma_1 \to 0$, with the eigenfunctions concentrating on the other part of the domain $\Omega$.

\begin{theorem}
\label{thm:gamma1.to.0}
Let  $\Lambda_1(\gamma_1)$ be the first eigenvalue of \eqref{eq:main.short} and ${\bf \Phi}_{\gamma_1}$ the associated eigenfunction, with $\gamma_2$ fixed. Then
 \begin{equation}
     \label{eq:gamma0min}
 \Lambda_0:=\lim_{\gamma_1\to 0} \Lambda_1(\gamma_1)=\min\{
    \sigma_1^{\Omega_1}[-\Delta+c_1;\mathcal{N}]; \sigma_1^{\Omega_2}[-\Delta+c_2;\mathcal{N}+\gamma_2]
    \}.
     \end{equation}
Furthermore, if ${\bf \Phi}_{\gamma_1}$ is such that $\|{\bf \Phi}_{\gamma_1}\|_{\mathcal{H}^1}=1$, then
$$
\lim_{\gamma_1\to 0}{\bf \Phi}_{\gamma_1}={\bf \Phi}_0\quad\text{in $\mathcal{H}^1$,}
$$
where
    $$
        {\bf \Phi}_0=(\varphi_{1,0},\varphi_{2,0})=\left\{\begin{array}{ll}
            (v_{1,0},v_{2,0})& \text{if } \Lambda_0=\sigma_1^{\Omega_1}[-\Delta+c_1;\mathcal{N}]<\sigma_1^{\Omega_2}[-\Delta+c_2;\mathcal{N}+\gamma_2],\\[6pt]
            (0,w_{2,0})& \text{if } \Lambda_0=\sigma_1^{\Omega_2}[-\Delta+c_2;\mathcal{N}+\gamma_2]\leq \sigma_1^{\Omega_1}[-\Delta+c_1;\mathcal{N}]
        \end{array}\right.
    $$
    with $v_{1,0}$ and $w_{2,0}$ being the positive eigenfunctions associated with $\sigma_1^{\Omega_1}[-\Delta+c_1;\mathcal{N}]$ and $\sigma_1^{\Omega_2}[-\Delta+c_2;\mathcal{N}+\gamma_2]$, {with $\|v_{1,0}\|_{H^1(\Omega_1)}=1$ and $\|w_{2,0}\|_{H^1(\Omega_2)}=1$}, respectively, and $v_{2,0}$ the unique positive solution of
\begin{equation}
    \label{ecu02}
\left\{\begin{array}{ll}
(-\Delta+c_2(x)) \varphi_{2} =\Lambda_0 \varphi_2,\quad   &\hbox{in } \Omega_2,\\
\partial_{\nu_2} \varphi_2+\gamma_2 \varphi_2 =\gamma_2\varphi_{1,0},&\hbox{on } \Sigma,\\
\partial_{\nu_2} \varphi_2=0,&\hbox{on } \Gamma.
 \end{array}\right.
\end{equation}
\end{theorem}

\begin{proof}
First, note that, by elliptic regularity, for any $p>1$ there exists a constant $C(p) > 0$, independent of $\gamma_1$, such that
$$
 \|{\bf \Phi}_{\gamma_1}\|_{\mathcal{W}^{2,p}}\leq C,
 $$
 Fixing $p > N$, the compact embedding $\mathcal{W}^{2,p}(\Omega) \subset \mathcal{C}^1(\bar{\Omega})$ implies that, up to a subsequence,
 $$
 {\bf \Phi}_{\gamma_1}\to {\bf \Phi}_{0}\quad\text{in $\mathcal{C}^1$,}
 $$
 for some ${\bf \Phi}_{0}\geq 0$, $\|{\bf \Phi}_{0}\|_{\mathcal{H}^1}=1$. Hence, passing to the limit in the weak formulation
\eqref{eq:weak.formulation}, first with
${\bf\Psi}=(\psi_1,0)$, we obtain
\begin{equation}
\label{eq:limit.gamma0.1}
\int_{\Omega_1}
\nabla\varphi_{1,0}\cdot\nabla\psi_1
+
\int_{\Omega_1}
c_1(x)\varphi_{1,0}\psi_1
=
\Lambda_0
\int_{\Omega_1}\varphi_{1,0}\psi_1
\end{equation}
for every $\psi_1\in H^1(\Omega_1)$.
Similarly, taking ${\bf\Psi}=(0,\psi_2)$ gives
\begin{equation}
\label{eq:limit.gamma0.2}
\int_{\Omega_2}
\nabla\varphi_{2,0}\cdot\nabla\psi_2
+
\int_{\Omega_2}
c_2(x)\varphi_{2,0}\psi_2
+
\gamma_2\int_{\Sigma}
(\varphi_{2,0}-\varphi_{1,0})\psi_2
=
\Lambda_0
\int_{\Omega_2}\varphi_{2,0}\psi_2
\end{equation}
for every $\psi_2\in H^1(\Omega_2)$.

To identify the limit function ${\bf \Phi}$ and $\Lambda_0$ we distinguish three cases according to the minimum in \eqref{eq:gamma0min}.

\noindent{\sc Case 1}:
 $ \sigma_1^{\Omega_2}[-\Delta+c_2;\mathcal{N}+\gamma_2]<\sigma_1^{\Omega_1}[-\Delta+c_1;\mathcal{N}]$.

By the monotonicity of the eigenvalues with respect to the boundary conditions, see \eqref{eq:order.eigenvalues.scalar}, and by Lemma~\ref{autovalores_acotados}, we have that
$$
\Lambda_1(\gamma_1)<\sigma_1^{\Omega_2}[-\Delta+c_2;\mathcal{N}+\gamma_2]<\sigma_1^{\Omega_1}[-\Delta+c_1;\mathcal{N}]<\sigma_1^{\Omega_1}[-\Delta+c_1;\mathcal{N}+\gamma_1].
$$

Now, assume first that $\varphi_{1,0}\not\equiv 0$. Then,  from \eqref{eq:limit.gamma0.1} we have that $\varphi_{1,0}$ is a non-negative eigenfunction
\begin{equation}
\label{eq:Neumann_prob_Lam_0}
\left\{\begin{array}{ll}
(-\Delta+c_1(x)) \varphi_{1,0} =\Lambda_0 \varphi_{1,0},\quad   &\hbox{in } \Omega_1,\\
\partial_{\nu}\varphi_{1,0} =0,&\hbox{on } \partial\Sigma.
 \end{array}\right.
 \end{equation}
 Therefore,
$$
\Lambda_0= \sigma_1^{\Omega_1}[-\Delta+c_1;\mathcal{N}],
$$
which is a contradiction. Thus, we conclude that
$$
\varphi_{1,0}=0,
$$
and therefore, since $\|{\bf \Phi}_0\|_{\mathcal{H}^1}=1,$ we have $\varphi_{2,0}\not\equiv 0$. Now, \eqref{eq:limit.gamma0.2} reduces to
$$
   \int_{\Omega_2} \nabla \varphi_2 \cdot\nabla \psi_2 +
    \int_{\Omega_2} c_2(x) \varphi_2 \psi_2  +\gamma_2\int_{\Sigma} \varphi_2\psi_2
    =\Lambda_0 \int_{\Omega_2} \varphi_2 \psi_2.
$$
Hence, ${\varphi_{2,0}}$ is an eigenfunction of the uncoupled Robin problem and, by simplicity of the eigenvalues,
$$
\Lambda_0:=\sigma_1^{\Omega_2}[-\Delta+c_2;\mathcal{N}+\gamma_2].
$$

\noindent {\sc Case 2}: $ \sigma_1^{\Omega_1}[-\Delta+c_1;\mathcal{N}]<\sigma_1^{\Omega_2}[-\Delta+c_2;\mathcal{N}+\gamma_2]$.

By continuity of the eigenvalue with respect to the permeability parameter,  see Lemma~\ref{lemma:cont.monotone}, we have that there exists $\gamma_1$ sufficiently small so that
$$
\sigma_1^{\Omega_1}[-\Delta+c_1;\mathcal{N}]<\sigma_1^{\Omega_1}[-\Delta+c_1;\mathcal{N}+\gamma_1]<\sigma_1^{\Omega_2}[-\Delta+c_2;\mathcal{N}+\gamma_2].
$$
Hence, by Lemma~\ref{autovalores_acotados} we have
$$
\Lambda_1(\gamma_1)<\sigma_1^{\Omega_1}[-\Delta+c_1;\mathcal{N}+\gamma_1]<\sigma_1^{\Omega_2}[-\Delta+c_2;\mathcal{N}+\gamma_2].
$$

We claim now that $\varphi_{1,0}\not\equiv0$. Indeed, otherwise
$\varphi_{2,0}\not\equiv0$, since
$\|{\bf\Phi}_0\|_{\mathcal H^1}=1$. Furthermore,~\eqref{eq:limit.gamma0.2} would reduce to the homogeneous Robin
eigenvalue problem in $\Omega_2$. Thus
$$
\Lambda_0=\sigma_1^{\Omega_2}[-\Delta+c_2;\mathcal{N}+\gamma_2],
$$
a contradiction.

Hence $\varphi_{1,0}\not\equiv0$, and
\eqref{eq:limit.gamma0.1}, together with
$\varphi_{1,0}\geq0$, implies that $\varphi_{1,0}$ is the principal eigenfunction, up to a multiplicative constant, associated with
$$
\Lambda_0=\sigma_1^{\Omega_1}[-\Delta+c_1;\mathcal{N}].
$$

Now to identify $\varphi_{2,0}$, we note that from  \eqref{eq:limit.gamma0.2} we have that  $\varphi_{2,0}$ solves the problem
$$
\left\{
\begin{array}{ll}
(-\Delta+c_2(x)-\Lambda_0)\varphi_{2,0}
=0,
&\text{in }\Omega_2,\\
\partial_{\nu_2}\varphi_{2,0}
+\gamma_2\varphi_{2,0}
=\gamma_2\varphi_{1,0},
&\text{on }\Sigma,\\
\partial_{\nu_2}\varphi_{2,0}=0,
&\text{on }\Gamma.
\end{array}
\right.
$$
The
above problem admits a unique positive solution $\varphi_{2,0}$, for such mixed boundary conditions; see \cite{Santi}.

\noindent{\sc Case 3}:  $ \sigma_1^{\Omega_2}[-\Delta+c_2;\mathcal{N}+\gamma_2]
=\sigma_1^{\Omega_1}[-\Delta+c_1;\mathcal{N}]$.

In this limit case, we find that the behavior is the same as in {\sc Case 1}. Thus, we find that $\varphi_{1,0}=0$, $\varphi_{2,0}>0$ and we conclude that $\Lambda_0=\sigma_1^{\Omega_2}[-\Delta+c_2;\mathcal{N}+\gamma_2]$.

To prove it we first claim that $\varphi_{1,0}=0$ in $\Omega_1$
and argue by contradiction. Indeed, suppose that
$\varphi_{1,0}\not\equiv0$. Then, thanks to
\eqref{eq:limit.gamma0.1} and
$\varphi_{1,0}\geq 0$ it follows that
$\varphi_{1,0}$ is the associated principal eigenfunction with the principal eigenvalue $\sigma_1^{\Omega_1}[-\Delta+c_1;\mathcal{N}]$ of the Neumann problem \eqref{eq:Neumann_prob_Lam_0}.
{
Taking in  \eqref{eq:weak.formulation}
${\bf\Psi}=(0,w_{2,0})$, we obtain
$$
\int_{\Omega_2}
\nabla\varphi_{2,\gamma_1}\cdot\nabla w_{2,0}
+
\int_{\Omega_2}
c_2(x)\varphi_{2,\gamma_1}w_{2,0}
+
\gamma_2\int_{\Sigma}
(\varphi_{2,\gamma_1}-\varphi_{1,\gamma_1})w_{2,0}
=
\Lambda_1(\gamma_1)
\int_{\Omega_2}\varphi_{2,\gamma_1}w_{2,0}.
$$
On the other hand, since $w_{2,0}$ is the positive eigenfunction associated with $\sigma_1^{\Omega_2}[-\Delta+c_2;\mathcal{N}+\gamma_2]$, using the weak formulation of that problem for the test function $\varphi_{2,\gamma_1}$ we have
\[
\int_{\Omega_2}
\nabla w_{2,0}\cdot\nabla\varphi_{2,\gamma_1}
+
\int_{\Omega_2}
c_2(x)w_{2,0}\varphi_{2,\gamma_1}
+
\gamma_2\int_{\Sigma}
w_{2,0}\varphi_{2,\gamma_!}
=
\Lambda_0
\int_{\Omega_2}w_{2,0}\varphi_{2,\gamma_1}.
\]
Now, subtracting both identities, and passing to the limit,  we get
\[
\gamma_2\int_{\Sigma}
\varphi_{1,0}w_{2,0}=0.
\]
However, that is not possible since $\gamma_2>0$ and both $\varphi_{1,0}$ and $w_{2,0}$ are positive on $\Sigma$.
Therefore, one of them must be zero, proving this case.
}
\end{proof}

\subsection{A completely permeable interface: $\gamma_1\to \infty$}
As we will show below, the previous case, when there is no permeability, shows a similar behavior while assuming symmetric or asymmetric permeability. However, the case we are dealing with in this section provides us again with a very interesting result, showing a new and different behavior from the symmetric case.

\begin{theorem}
\label{thm:gamma1.to.infty}Let  $\Lambda_1(\gamma_1)$ be the first eigenvalue of \eqref{eq:main.short} and ${\bf \Phi}_{\gamma_1}$ the associated eigenfunction. Then
    $$\Lambda_\infty:=\lim_{\gamma_1\to\infty} \Lambda(\gamma_1)=\min\{
    \sigma_1^{\Omega_1}[-\Delta+c_1;\mathcal{D}]; \sigma_1^{\Omega_2}[-\Delta+c_2;\mathcal{N}]
    \}.$$
Furthermore, if ${\bf \Phi}_{\gamma_1}$ is such that $\|{\bf \Phi}_{\gamma_1}\|_{\mathcal{H}^1}=1$, then
$$
\lim_{\gamma_1\to \infty}{\bf \Phi}_{\gamma_1}={\bf \Phi}_\infty\quad\text{in $\mathcal{H}^1$,}
$$
with
    $$
        {\bf \Phi}_\infty=(\varphi_{1,\infty},\varphi_{2,\infty})=\left\{\begin{array}{ll}
            (v_{1,\infty},0)& \text{if } \Lambda_\infty=\sigma_1^{\Omega_1}[-\Delta+c_1;\mathcal{D}]\leq \sigma_1^{\Omega_2}[-\Delta+c_2;\mathcal{N}]\\[6pt]
            (w_{1,\infty},w_{2,\infty})& \text{if } \Lambda_\infty=\sigma_1^{\Omega_2}[-\Delta+c_2;\mathcal{N}]<\sigma_1^{\Omega_1}[-\Delta+c_1;\mathcal{D}]
        \end{array}\right.
    $$
  with $v_{1,\infty}$ and $w_{2,\infty}$ being the positive eigenfunctions associated to $\sigma_1^{\Omega_1}[-\Delta+c_1;\mathcal{D}]$ and $\sigma_1^{\Omega_2}[-\Delta+c_2;\mathcal{N}]$ {with $\|v_{1,\infty}\|_{H^1(\Omega_1)}=1$ and $\|w_{2,\infty}\|_{H^1(\Omega_2)}=1$}, respectively, and $w_{1,\infty}$ the unique positive solution of
$$
\left\{\begin{array}{ll}
(-\Delta+c_1(x)-\Lambda_\infty) \varphi_{1} =0\quad   &\hbox{in } \Omega_1,\\
\varphi_1 =\varphi_{2,\infty},&\hbox{on } \Sigma.
 \end{array}\right.
$$
\end{theorem}

The proof of this theorem follows the steps of the scalar case; however, due to the lack of symmetry, it requires auxiliary results that we prove next.

{The first result provides us with a bound for the interface term. One main difficulty we face in obtaining some results of this paper is that the difference $(\varphi_{2,n}-\varphi_{1,n})$ has no defined sign on $\Sigma$. Even more, we have to control the boundary term when $\gamma_1\to\infty$. For these reasons, the next result plays a crucial role in the sequel.}
 \begin{lemma}
\label{lemma:bound.gamma.L2}
There exists a positive constant $C>0$, independent of $n$, such that
    $$\gamma_{1,n} \int_{\Sigma}|\varphi_{2,n}-\varphi_{1,n}|^2\leq C.$$
\end{lemma}

\begin{proof}
    Let $U\subset\mathbb{R}^N$ be a bounded domain and consider, for $f\in L^2(U)$ and $g\in H^1(U)\cap C(\overline{D})$, the problem
    \begin{equation}
        \label{eq:levantamiento}
        \left\{\begin{array}{ll}
           -\Delta e=f,\qquad  & \text{ in } U,  \\
             e=g,& \text{ in } \partial U.
        \end{array}\right.
    \end{equation}
Consider also the space of test functions $\mathcal{K}:=\{v\in H^1(U) : v-g\in H^1_0(U)\}$. Under these conditions,  there exists a unique weak solution, $e\in \mathcal{K}$, of~\eqref{eq:levantamiento}
$$
\int_{D} \nabla e\cdot \nabla v=\int_{D} fv, \quad\text{for all }v\in\mathcal{K},
$$
that, moreover, verifies
    \begin{equation}
        \label{eq:levantamiento.cota}
\|e\|_{H^1(U)}\leq K(\|f\|_2+\|g\|_{H^1(U)}),
    \end{equation}
for some constant $K>0$, see~\cite[Proposition IX.22]{Brezis.book},~\cite[Theorem 8.3]{GilbargTrudinger.book}.

Now consider problem~\eqref{eq:levantamiento} with $f=0$ and $U=\Omega_1$ and $g=\widetilde{\varphi}_{2,n}-\varphi_{1,n}$, where $\widetilde{\varphi}_{2,n}$ is the extension to $\Omega_1$ of $\varphi_{2,n}$. Moreover, $\widetilde{\varphi}_{2,n}$ verifies
$$
\|\tilde{\varphi}_{2,n}\|_{H^1(\Omega)}\leq
\|\varphi_{2,n}\|_{H^1(\Omega_2)},
$$
see~\cite[Theorem 1, Section 5.4]{Evans}.
Then, from~\eqref{eq:levantamiento.cota}, we get
$$
\|e_n\|_{H^1(\Omega_1)}\leq K\|\tilde{\varphi}_{2,n}-\varphi_{1,n}\|_{H^1(\Omega_1)}\leq C(\|\varphi_{1,n}\|_{H^1(\Omega_1)}+\|{\varphi}_{2,n}\|_{H^1(\Omega_2)})=C,
$$
and the last inequality is true thanks to the normalization (\ref{eq:normalization}).

Going back the principal eigenfunction ${\bf \Phi}_n$, taking as test function
$\Psi=(-e_n,0)\in \mathcal{H}^1$ in~\eqref{eq:weak.formulation}  yields
$$
-\int_{\Omega_1} \nabla \varphi_{1,n}\cdot \nabla e_n + \gamma_{1,n} \int_{\Sigma} (\varphi_{2,n}-\varphi_{1,n})^2=\Lambda_n \int_{\Omega_1} \varphi_{1,n} e_n.
$$
Finally, since $e_n$ is bounded in $H^1(\Omega_1)$, we have that
$$
-\int_{\Omega_1} \nabla \varphi_{1,n}\cdot \nabla e_n,\quad \text{and}\quad \Lambda_n \int_{\Omega_1} \varphi_{1,n} e_n
$$
are bounded (independently of $n$). Hence
$$\gamma_{1,n} \int_{\Sigma} (\varphi_{2,n}-\varphi_{1,n})^2$$
is uniformly bounded, independent of $n$.
\end{proof}

\begin{corollary}
    \label{corol:L2.to0}
Let $\gamma_{1,n}\to+\infty$ as $n\to\infty$. Then
$$
\|\varphi_{2,n}-\varphi_{1,n}\|_{L^2(\Sigma)}
=O\left(\gamma_{1,n}^{-1/2}\right).
$$
In particular,
$$(\varphi_{2,n}-\varphi_{1,n})\to0
\qquad\text{strongly in }L^2(\Sigma)
\quad\text{as }n\to\infty.
$$
\end{corollary}

\begin{proof}
By Lemma~\ref{lemma:bound.gamma.L2},
$$
\gamma_{1,n}
\|\varphi_{2,n}-\varphi_{1,n}\|_{L^2(\Sigma)}^2
\le C,
$$
where $C>0$ is independent of $n$. Hence,
$$
\|\varphi_{2,n}-\varphi_{1,n}\|_{L^2(\Sigma)}
\le C^{1/2}\gamma_{1,n}^{-1/2},
$$
which proves the assertion.
\end{proof}

\begin{lemma}
\label{lemma.convergence.phi}
There exists ${\bf \Phi}_\infty=(\varphi_{1,\infty},\varphi_{2,\infty})\in \mathcal{H}^1$ such that, up to a subsequence,
$$
\varphi_{i,n}\rightarrow \varphi_{i,\infty}
\quad\text{strongly in } H^1(\Omega_i),
$$
as $n\to\infty$. Moreover,
$$
\|\varphi_{1,\infty}\|_{H^1(\Omega_1)}
+\|\varphi_{2,\infty}\|_{H^1(\Omega_2)}=1.
$$
\end{lemma}

\begin{remark}
\label{rem:not0}
 Note that the limit functions cannot be zero simultaneously: $
(\varphi_{1,\infty},\varphi_{2,\infty})\neq(0,0).
$
\end{remark}

\begin{proof}
Since
$
\|{\bf \Phi}_n\|_{\mathcal H^1}=1,
$
the sequence $\{{\bf \Phi}_n\}$ is bounded in $\mathcal H^1$. Hence, up to a
subsequence, there exists ${\bf \Phi}_\infty
=(\varphi_{1,\infty},\varphi_{2,\infty})\in\mathcal H^1
$
such that
$$
{\bf \Phi}_n\rightharpoonup{\bf \Phi}_\infty
\qquad\text{weakly in }\mathcal H^1.
$$
Moreover, by the compact embedding
$H^1(\Omega_i)\hookrightarrow L^2(\Omega_i)$ and the compactness of the
trace operator from $H^1(\Omega_i)$ into $L^2(\Sigma)$, we have
\begin{equation}
    \label{eq:strong.convergence}
    \varphi_{i,n}\rightarrow\varphi_{i,\infty}
\qquad\text{strongly in $L^2(\Omega_i)$ and in $L^2(\Sigma)$}.
\end{equation}

We first prove the strong convergence of $\{\varphi_{2,n}\}$ in
$H^1(\Omega_2)$. Recall that $\varphi_{2,n}$ satisfies the non-homogeneous Robin type problem
$$
\left\{
\begin{array}{ll}
(-\Delta+c_2(x))\varphi_{2,n}
   =\Lambda_n\varphi_{2,n},
   & \text{in }\Omega_2,\\[1mm]
\partial_{\nu_2}\varphi_{2,n}+\gamma_2\varphi_{2,n}
   =\gamma_2\varphi_{1,n},
   & \text{on }\Sigma,\\[1mm]
\partial_{\nu_2}\varphi_{2,n}=0,
   & \text{on }\Gamma.
\end{array}
\right.
$$
By standard elliptic regularity, see~\cite{Amann1993},
we have that there exists a constant $C>0$, independent of $n$, such
that
    $$
        \|\varphi_{2,n}\|_{W^{2,2}(\Omega_2)}\leq C\|\varphi_{2,n}-\varphi_{1,n}\|_{H^{1/2}(\Sigma)}\leq C(\|\varphi_{2,n}\|_{H^1}+\|\varphi_{1,n}\|_{H^{1}})\leq C.
    $$
Therefore, using the compact embedding
$
W^{2,2}(\Omega_2)\hookrightarrow H^1(\Omega_2),
$
we obtain, up to a subsequence, that there exists $\varphi_{2,\infty}\in H^1(\Omega_2)$, $\varphi_{2,\infty}\geq 0$, such that
$$
\varphi_{2,n}\rightarrow\varphi_{2,\infty}
\qquad\text{strongly in }H^1(\Omega_2).
$$

We now prove the strong convergence of $\{\varphi_{1,n}\}$ in
$H^1(\Omega_1)$. Since
$$
\varphi_{1,n}\rightarrow\varphi_{1,\infty}
\qquad\text{strongly in }L^2(\Omega_1),
$$
it is enough to prove that $\{\nabla\varphi_{1,n}\}$ is a Cauchy sequence
in $L^2(\Omega_1)$. Let $n,m\in\mathbb N$ and set
$$
w_{n,m}:=\varphi_{1,n}-\varphi_{1,m}.
$$
Subtracting the equations satisfied by $\varphi_{1,n}$ and
$\varphi_{1,m}$ and testing the resulting equation with $w_{n,m}$, we
obtain
$$
\begin{aligned}
\int_{\Omega_1}|\nabla w_{n,m}|^2
={}&
-\int_{\Omega_1}c_1(x)w_{n,m}^2
-\int_\Sigma
\Big[
\gamma_{1,m}(\varphi_{2,m}-\varphi_{1,m})
-\gamma_{1,n}(\varphi_{2,n}-\varphi_{1,n})
\Big]w_{n,m}
\\
&+\int_{\Omega_1}
\big(
\Lambda_n\varphi_{1,n}
-\Lambda_m\varphi_{1,m}
\big)w_{n,m}.
\end{aligned}
$$
We estimate the terms on the right-hand side separately. Since
$c_1\in L^\infty(\Omega_1)$ and
$w_{n,m}\to0$ in $L^2(\Omega_1)$ as $n,m\to\infty$,
$$
\left|
\int_{\Omega_1}c_1(x)w_{n,m}^2
\right|
\leq
\|c_1\|_{L^\infty(\Omega_1)}
\|w_{n,m}\|_{L^2(\Omega_1)}^2
\to 0.
$$

For the eigenvalue term, we have that
    $$
   \left| \int_{\Omega_1} (\Lambda_n \varphi_{1,n}-
    \Lambda_m\varphi_{1,m})w_{n,m}\right| \leq
 |\Lambda_n|
 \|w_{n,m}\|_{L^2(\Omega_1)}^2
 +
 |\Lambda_n-\Lambda_m|
 \|\varphi_{1,m}\|_{L^2(\Omega_1)}
 \|w_{n,m}\|_{L^2(\Omega_1)}\to 0
    $$
    since, on one hand $\{\varphi_{1,n}\}$ is bounded in
$L^2(\Omega_1)$  and on the other $\Lambda_n,\Lambda_m\to \Lambda_\infty $, by Lemma~\ref{lemma:Lambda.convergence}.

Finally, consider the interface term; then, due to Lemma~\ref{lemma:bound.gamma.L2},
$$
\gamma_{1,k}
\|\varphi_{2,k}-\varphi_{1,k}\|_{L^2(\Sigma)}
\leq C,
\qquad k=m,n.
$$
Thus, for $k=m,n$
$$
\left|
\int_\Sigma
\gamma_{1,k}(\varphi_{2,k}-\varphi_{1,k})
w_{n,m}
\right|
\leq
\gamma_{1,k}
\|\varphi_{2,k}-\varphi_{1,k}\|_{L^2(\Sigma)}
\|w_{n,m}\|_{L^2(\Sigma)}
\leq
C\|w_{n,m}\|_{L^2(\Sigma)}
\rightarrow0,
$$ because $\varphi_{1,k}\rightarrow\varphi_{1,\infty}$
strongly in $L^2(\Sigma)$, see~\eqref{eq:strong.convergence}.

Consequently, we conclude that
    $$
        \int_{\Omega_1} |\nabla \varphi_{1,n}-\nabla \varphi_{1,m}|^2 \to 0, \text{\qquad as\ }m, n\to\infty,
    $$
which together with
$$
\|\varphi_{1,n}-\varphi_{1,m}\|_{L^2(\Omega_1)}
\rightarrow0,
$$
shows that $\{\varphi_{1,n}\}$ is a Cauchy sequence in
$H^1(\Omega_1)$ and therefore
$$
\varphi_{1,n}\rightarrow\varphi_{1,\infty}
\qquad\text{strongly in }H^1(\Omega_1).
$$

Finally, passing to the limit in the normalization
\eqref{eq:normalization}, we obtain
$$
\|\varphi_{1,\infty}\|_{H^1(\Omega_1)}
+
\|\varphi_{2,\infty}\|_{H^1(\Omega_2)}
=1,
$$
and in particular, $(\varphi_{1,\infty},\varphi_{2,\infty})\neq(0,0).
$
\end{proof}

Now, to finish the characterization of the limit behavior when $\gamma_1\to\infty$, we have just to identify the limits $\Lambda_\infty$ and ${\bf \Phi}_\infty.$

\begin{proof}[Proof of Theorem~\ref{thm:gamma1.to.infty}]
As in the proof of Theorem~\ref{thm:gamma1.to.0}, we distinguish three cases according to the relation between the limit eigenvalues, $\sigma_1^{\Omega_1}[-\Delta+c_1;\mathcal{D}]$ and $\sigma_1^{\Omega_2}[ -\Delta+c_2;\mathcal{N}].$

\noindent{\sc Case 1}:
$\sigma_1^{\Omega_1}[-\Delta+c_1;\mathcal{D}]< \sigma_1^{\Omega_2}[ -\Delta+c_2;\mathcal{N}].$

Recall first that, by Corollary~\ref{corol:Lambda.less.D}, we have that
$$
\Lambda_1(\gamma_1)<\sigma_1^{\Omega_1}[-\Delta+c_1;\mathcal{D}], \qquad\forall \gamma_1>0.
 $$
 Then
\begin{equation}
    \label{primera}
\lim_{\gamma_1\to +\infty}\Lambda_1(\gamma_1)\leq \sigma_1^{\Omega_1}[-\Delta+c_1;\mathcal{D}]<\sigma_1^{\Omega_2}[ -\Delta+c_2;\mathcal{N}].
\end{equation}

First, we claim that $\varphi_{2,\infty}=0$ in $\Omega_2$. Indeed, assume by contradiction that $\varphi_{2,\infty}\not\equiv 0$. Then,  taking ${\bf\Psi}=(0,\psi_2)\in\mathcal{H}^1$ as test function in~\eqref{eq:weak.formulation}, we get
$$
    \int_{\Omega_2} \nabla \varphi_2 \cdot \nabla \psi_2 +
     \int_{\Omega_2} c_2(x) \varphi_2 \psi_2  +\gamma_2\int_{\Sigma} (\varphi_2-\varphi_1)\psi_2=\Lambda(\gamma_1) \int_{\Omega_2} \varphi_2 \psi_2
$$
We can pass to the limit, at least along a subsequence, using Corollary~\ref{corol:L2.to0} and Lemma~\ref{lemma.convergence.phi}, and get for all ${\bf \Psi}\in\mathcal{H}^1$
\begin{equation}
    \label{eq:limit.phi2}
    \int_{\Omega_2} \nabla \varphi_{2,\infty}\cdot \nabla \psi_2 +
     \int_{\Omega_2} c_2(x) \varphi_{2,\infty} \psi_2  =\Lambda_\infty \int_{\Omega_2} \varphi_{2,\infty} \psi_2,
\end{equation}
where $\Lambda_\infty$ is the limit of $\Lambda(\gamma_1)$, see Lemma~\ref{lemma:Lambda.convergence}. This implies that $\varphi_{2,\infty}$ is the weak solution of
\begin{equation}
    \label{eq:Neumann_prob}
\left\{\begin{array}{ll}
(-\Delta+c_2(x)) \varphi_{2,\infty} =\Lambda_\infty \varphi_{2,\infty},\quad   &\hbox{in } \Omega_2,\\
\partial_{\nu}\varphi_{2,\infty} =0,&\hbox{on } \partial\Omega_2,
 \end{array}\right.
\end{equation}
and hence $\Lambda_\infty=\sigma_1^{\Omega_2}[ -\Delta+c_2;\mathcal{N}]$, a contradiction with~\eqref{primera}.
So that, as claimed, $\varphi_{2,\infty}=0$ in $\Omega_2$ and $\varphi_{2,\infty}\in H^1(\Omega)$. Now, the trace theorem yields $\varphi_{2,\infty}=0$ on $\partial \Omega_2$ and in particular on $\Sigma$.

Therefore, due to Remark~\ref{rem:not0}, $\varphi_{1,\infty}\not\equiv 0$ in $\Omega_1$. Moreover, taking ${\bf\Psi}=(\psi_1,\psi_2)\in\mathcal{H}^1$ as test function in~\eqref{eq:weak.formulation}, with $\psi_1\in H_0^1(\Omega_1)$, and passing to the limit  we get,
$$\int_{\Omega_1} \nabla \varphi_{1,\infty} \cdot\nabla \psi_1 +
     \int_{\Omega_1} c_1(x) \varphi_{1,\infty} \psi_1
            +\gamma_2\int_{\Sigma} (\varphi_{2,\infty}-\varphi_{1,\infty})\psi_2=\Lambda_\infty \int_{\Omega_1} \varphi_{1,\infty} \psi_1.
  $$
Hence, since by~Corollary~\ref{corol:L2.to0} we have that $\varphi_{1,\infty}=\varphi_{2,\infty}=0$ a.e on $\Sigma$,
 by an analogous argument as the one done in~\eqref{eq:limit.phi2} we get that  $\varphi_{1,\infty}$ is the first eigenfunction of
\begin{equation}
\label{eq:dirchlet_prob}
\left\{\begin{array}{ll}
(-\Delta+c_1(x)) \varphi_{1,\infty} =\Lambda_\infty \varphi_{1,\infty}\quad   &\hbox{in } \Omega_1,\\
\varphi_{1,\infty} =0,&\hbox{on } \Sigma.
 \end{array}\right.
\end{equation}
Therefore, we conclude that $\Lambda_\infty=\sigma_1^{\Omega_1}[ -\Delta+c_1;\mathcal{D}]$ and ${\bf \Phi}_\infty=(\varphi_{1,\infty},0)$.

\noindent{\sc Case 2}:
$\sigma_1^{\Omega_2}[ -\Delta+c_2;\mathcal{N}]<\sigma_1^{\Omega_1}[-\Delta+c_1;\mathcal{D}].$

First, from~\eqref{eq:cota.Lambda.autovalor}  we have that
\begin{equation}
    \label{primerabis}
    \lim_{\gamma_1\to +\infty}\Lambda_1(\gamma_1)\leq\lim_{\gamma_1\to +\infty}\frac{a_1+\frac{\gamma_1}{\gamma_2}\sigma_1^{\Omega_2}[-\Delta +c_2,\mathcal{N}]}{b_1+\frac{\gamma_1}{\gamma_2}}\leq \s_1^{\Omega_2}[-\Delta +c_2,{\mathcal{N}}]
\end{equation}
and hence
$$
\Lambda_\infty\leq \sigma_1^{\Omega_2}[ -\Delta+c_2;\mathcal{N}]<\sigma_1^{\Omega_1}[-\Delta+c_1;\mathcal{D}].
$$

To analyze the limit behavior of the eigenfunctions, we claim in this case that   $\varphi_{2,\infty}\neq 0$ in $\Omega_2$. Indeed, if we assume that $\varphi_{2,\infty}=0$, then necessarily $\varphi_{1,\infty}\neq 0$. Testing now in~\eqref{eq:weak.formulation} with the function $(\psi_1,0)$, where  $\psi_1\in H_0^1(\Omega_1)$ and arguing as before, we conclude that $ \Lambda_\infty= \sigma_1^{\Omega_1}[-\Delta+c_1;\mathcal{D}]$, a contradiction.
Hence $\varphi_{2,\infty}\neq  0$ is the first eigenfunction of the Neumann uncoupled problem and $\Lambda_\infty=\sigma_1^{\Omega_2}[ -\Delta+c_2;\mathcal{N}]$ the corresponding eigenvalue.

To identify $\varphi_{1,\infty}$ we take again the test function $(\psi_1,0)$ with $\psi_1\in H_0^1(\Omega_1)$, which allows to get rid of the integral in $\Sigma$,
and pass to the limit to obtain
$$ \int_{\Omega_1} \nabla \varphi_{1,\infty}\cdot \nabla \psi_1 +
     \int_{\Omega_1} c_1(x) \varphi_{1,\infty} \psi_1  =\Lambda_\infty \int_{\Omega_1} \varphi_{1,\infty} \psi_1.
$$
Having that $\varphi_{1,\infty}=\varphi_{2,\infty}>0$ across the interface, this implies that $\varphi_{1,\infty}$ is a weak solution of
$$
\left\{\begin{array}{ll}
(-\Delta+c_1(x)-\Lambda_\infty) \varphi_{1} =0\quad   &\hbox{in } \Omega_1,\\
\varphi_1 =\varphi_{2,\infty},&\hbox{on } \Sigma,
 \end{array}\right.
$$
which exists, since $\sigma_1^{\Omega_1}[-\Delta+c_1-\Lambda_\infty;\mathcal{D}]=\sigma_1^{\Omega_1}[-\Delta+c_1;\mathcal{D}]-\Lambda_\infty>0$.

\noindent{\sc Case 3}:
$\sigma_1^{\Omega_2}[ -\Delta+c_2;\mathcal{N}]=\sigma_1^{\Omega_1}[-\Delta+c_1;\mathcal{D}].$

We first observe that thanks to Lemma\;\ref{autovalores_acotados} and Corollary\;\ref{corol:Lambda.less.D} we already know that $\Lambda_1(\gamma_1)\leq \sigma_1^{\Omega_2}[ -\Delta+c_2;\mathcal{N}]=\sigma_1^{\Omega_1}[-\Delta+c_1;\mathcal{D}]$. Now we are going to follow a similar argument as performed for the previous cases. Thus, we first claim that $
\varphi_{2,\infty}=0$. Arguing by contradiction we suppose the contrary, so that
$\varphi_{2,\infty}\not\equiv0$. Hence, passing to the limit in the
weak formulation with test functions of the form
${\bf\Psi}=(0,\psi_2)$, as in {\sc Case 1}, we obtain
\[
\int_{\Omega_2}
\nabla\varphi_{2,\infty}\cdot\nabla\psi_2
+
\int_{\Omega_2}
c_2(x)\varphi_{2,\infty}\psi_2
=
\Lambda_\infty
\int_{\Omega_2}\varphi_{2,\infty}\psi_2,
\qquad
\psi_2\in H^1(\Omega_2).
\]
Consequently, since $\varphi_{2,\infty}\geq0$ and
$\varphi_{2,\infty}\not\equiv0$, it follows that $\varphi_{2,\infty}$ is the associated principal eigenfunction, up to multiplicative constants, with the principal eigenvalue for the Neumann problem in $\Omega_2$ denoted by
\[
\Lambda_\infty
=
\sigma_1^{\Omega_2}[-\Delta+c_2;\mathcal N].
\]
{As a consequence, we also have $\Lambda_\infty = \sigma_1^{\Omega_1}[ -\Delta+c_1;\mathcal{D}]$. Moreover, since $\varphi_{2,\infty}$ is the principal eigenfunction of the homogeneous Neumann problem \eqref{eq:Neumann_prob}, i.e., $\varphi_{2,\infty}>0$ in $\Omega_2$, thanks to the Maximum Principle and Hopf's lemma if $\varphi_{2,\infty}$ reached the minimum of 0 on the boundary $\Sigma$ we would have $\partial_{\nu_2}\varphi_{2,\infty}<0$, contradicting the homogeneous boundary condition $\partial_{\nu_2}\varphi_{2,\infty}=0$. Thus, $\varphi_{2,\infty}>0$ on $\Sigma$.

On the other hand, by Corollary~\ref{corol:L2.to0}, the traces satisfy $\varphi_{1,\infty} = \varphi_{2,\infty} > 0$ a.e. on $\Sigma$. Next, taking test functions of the form $(\psi_1,0)$ with $\psi_1\in H_0^1(\Omega_1)$ and passing to the limit, we deduce that $\varphi_{1,\infty}$ satisfies the equation
\[
(-\Delta + c_1(x) - \Lambda_\infty)\varphi_{1,\infty} = 0 \quad \text{in } \Omega_1.
\]
Let $\varphi_\mathcal{D}> 0$ denote the principal eigenfunction associated with $\sigma_1^{\Omega_1}[ -\Delta+c_1;\mathcal{D}]$, satisfying $\varphi_\mathcal{D} = 0$ and $\partial_{\nu_1}\varphi_\mathcal{D} < 0$ on $\Sigma$ by Hopf's lemma. Since $\Lambda_\infty = \sigma_1^{\Omega_1}[ -\Delta+c_1;\mathcal{D}]$, integrating by parts gives
\[
0 = \int_{\Omega_1} \Big( \varphi_\mathcal{D} \Delta \varphi_{1,\infty} - \varphi_{1,\infty} \Delta \varphi_\mathcal{D} \Big) dx = - \int_{\Sigma} \varphi_{1,\infty} \partial_{\nu_1}\varphi_\mathcal{D} \, d\sigma.
\]
Since $\partial_{\nu_1}\varphi_\mathcal{D} < 0$ on $\Sigma$, this forces $\varphi_{1,\infty} = 0$ a.e. on $\Sigma$, which contradicts the fact that $\varphi_{1,\infty} = \varphi_{2,\infty} > 0$ a.e. on $\Sigma$. Therefore, we conclude that $\varphi_{2,\infty} \equiv 0$ in $\Omega_2$.}
\end{proof}

\subsection{A note on symmetric permeability}
Consider now the symmetric interface problem given by \begin{equation}
    \label{eq:system.symmetric}
\left\{
    \begin{array}
        {ll}
        -\Delta \varphi_i + c_i(x)
\varphi_i=\sigma \varphi_i&  \hbox{in } \Omega_i,\\
        \p_{\nu} \varphi_1=\p_{\nu} \varphi_2= \gamma(\varphi_2- \varphi_1),&\text{on\ }\Sigma,\\
        \p_{\nu} \varphi_2=0,&\text{on\ }\Gamma,
    \end{array}
\right.
\end{equation}
whose limit behavior has also been recently studied in~\cite{WangSu-preprint}.
As mentioned in the Introduction, in contrast with the non-symmetric case, when the permeability is  symmetric, the problem is self-adjoint, and the weak formulation reads just as:
\begin{equation}
 \label{eq:weak.formulation.symm}
\begin{aligned}
    \int_{\Omega_1} \nabla \varphi_1 \cdot\nabla \psi_1 &+ \int_{\Omega_2} \nabla \varphi_2 \cdot \nabla \psi_2 +
    \int_{\Omega_1} c_1(x) \varphi_1 \psi_1+
    \int_{\Omega_2} c_2(x) \varphi_2 \psi_2 \\& +\gamma\int_{\Sigma} (\varphi_2-\varphi_1)(\psi_2-\psi_1)
    =\Lambda \left(\int_{\Omega_1}\varphi_1 \psi_1 +\int_{\Omega_2} \varphi_2 \psi_2\right).
\end{aligned}
\end{equation}

\noindent{\bf Limit behavior when $\gamma\to 0$:}
We observe that the behavior is affected by the symmetry of the boundary condition on the membrane region, obtaining a different limit behavior from the asymmetric case, where such an asymmetric condition for the permeability parameters allowed the eigenfunctions to concentrate on both sizes of the membrane region, while in the symmetric case we face here this is not going to happen.

In getting this result, one can apply {the idea of the } proof performed in Theorem\;\ref{thm:gamma1.to.0} arriving at the stability of the eigenvalues $\Lambda_1(\gamma)$ and the associated eigenfunctions ${\bf \Phi}_{\gamma}$ of problem \eqref{eq:system.symmetric}, such that $\|{\bf \Phi}_{\gamma}\|_{\mathcal{L}^2}=1$.

{
\begin{theorem}
   \label{thm:symmetric.0}
Let  $\Lambda_1(\gamma)$ be the first eigenvalue of \eqref{eq:system.symmetric} and ${\bf \Phi}_{\gamma}$ the associated eigenfunction. Then
$$
\Lambda_0:= \lim_{\gamma\to 0}\Lambda(\gamma):=\min\{\sigma_1^{\Omega_1}[-\Delta+c_1;\mathcal{N}]; \sigma_1^{\Omega_2}[-\Delta+c_2;\mathcal{N}]\}.
$$
Furthermore, if ${\bf \Phi}_{\gamma}$ is such that $\|{\bf \Phi}_{\gamma}\|_{\mathcal{L}^2}=1$, then
$$
\lim_{\gamma\to 0}{\bf \Phi}_{\gamma}={\bf \Phi}_0\quad\text{in $\mathcal{H}^1$,}
$$
where
$$\Phi_0=
\left\{
\begin{array}{cc}
 (\varphi_{1,0},0)   & \text{if $\sigma_1^{\Omega_1}[-\Delta+c_1;\mathcal{N}]< \sigma_1^{\Omega_2}[-\Delta+c_2;\mathcal{N}]$, } \\
    (0,\varphi_{2,0})  & \text{if $\sigma_1^{\Omega_2}[-\Delta+c_2;\mathcal{N}]<\sigma_1^{\Omega_1}[-\Delta+c_1;\mathcal{N}]$, }
\end{array}\right.
$$
and $\varphi_{i,0}$ is the positive eigenfunction associated to  $\sigma_1^{\Omega_i}[-\Delta+c_i;\mathcal{N}]$, $i=1,2$ such that $\|\varphi_{i,0}\|_2=1$.
\end{theorem}

\begin{proof}
First, recall from Lemma~\ref{lemma:cont.monotone}, that $\Lambda(\gamma)$ is continuous and non-decreasing with $\gamma$. So, $\Lambda_0$ exists and by Lemma~\ref{autovalores_acotados} we have
$$
\Lambda_0\leq \min\{\sigma_1^{\Omega_1}[-\Delta+c_1;\mathcal{N}], \sigma_1^{\Omega_2}[-\Delta+c_2;\mathcal{N}]\}.
$$

 Moreover,  taking ${\bf \Phi}_\gamma=(\varphi_{1,\gamma},\varphi_{2,\gamma})$, normalized in $\mathcal{L}^2$,
as test function in~\eqref{eq:weak.formulation.symm} we get that
\begin{equation}
\label{cotanece}
\int_{\Omega_1} |\nabla \varphi_{1,\gamma}|^2 + \int_{\Omega_2} |\nabla \varphi_{2,\gamma}|^2 +
   \int_{\Omega_1} c_1(x) \varphi_{1,\gamma}^2 +    \int_{\Omega_2} c_2(x) \varphi_{2,\gamma}^2 +\gamma\int_{\Sigma} (\varphi_{2,\gamma}-\varphi_{1,\gamma})^2=\Lambda_1(\gamma).
\end{equation}
Hence, $\varphi_{i,\gamma}$ is bounded in $H^1(\Omega_i)$,  and there exists ${\bf \Phi}_0=(\varphi_{1,0},\varphi_{2,0})\in \mathcal{H}^1$, $\varphi_{i,0}\geq 0$ such that $\|{\bf \Phi}_0\|_{\mathcal{L}^2}=1$ and
\begin{equation}
   \varphi_{i,\gamma}\rightharpoonup \varphi_{i,0}\quad  \text{in $H^1(\O_i)$,}\qquad\text{and}\qquad
\varphi_{i,\gamma}\to \varphi_{i,0} \quad\text{in $L^2(\O_i)$ and $L^2(\Sigma)$}.
\end{equation}

Assume without loss of generality that $\sigma_1^{\Omega_1}[-\Delta+c_1;\mathcal{N}]<\sigma_1^{\Omega_2}[-\Delta+c_2;\mathcal{N}]$
and that by contradiction $\varphi_{2,0}\not\equiv 0$ in $\Omega_2$. Then,  taking ${\bf\Psi}=(0,\psi_2)\in\mathcal{H}^1$ as test function in the weak formulation~\eqref{eq:weak.formulation.symm}
we get
$$
   \int_{\Omega_2} \nabla \varphi_2 \cdot\nabla \psi_2 +
    \int_{\Omega_2} c_2(x) \varphi_2 \psi_2  +\gamma\int_{\Sigma} (\varphi_2-\varphi_1)\psi_2=\Lambda(\gamma) \int_{\Omega_2} \varphi_2 \psi_2.
$$
We can pass to the limit as $\gamma\to 0$ and use that
$(\varphi_2-\varphi_1)$ is bounded to get
and get for all $\Psi\in\mathcal{H}^1$
$$
   \int_{\Omega_2} \nabla \varphi_{2,0} \cdot \nabla \psi_2 +
    \int_{\Omega_2} c_2(x) \varphi_{2,0} \psi_2  =\Lambda_0 \int_{\Omega_2}
    \varphi_{2,0} \psi_2.
$$
This implies that $\varphi_{2,0}$ is the weak solution of
$$
\left\{\begin{array}{ll}
(-\Delta+c_2(x)) \varphi_{2,0} =\Lambda_0 \varphi_{2,0},\quad   &\hbox{in } \Omega_2,\\
\partial_{\nu}\varphi_{2

} =0,&\hbox{on } \partial\Omega_2,
\end{array}\right.
$$
and hence $\Lambda_0=\sigma_1^{\Omega_2}[ -\Delta+c_2;\mathcal{N}]$, a contradiction. Hence $\varphi_{2,0}=0$ in $\Omega_2$ and hence $\varphi_{1,0}\neq 0$ in $\Omega_1$.
Take now ${\bf\Psi}=(\psi_1, 0)\in\mathcal{H}^1$. Passing to the limit, we conclude that
$$
   \int_{\Omega_1} \nabla \varphi_{1,0} \cdot\nabla \psi_1 +
    \int_{\Omega_1} c_1(x) \varphi_{1,0} \psi_1  =\Lambda_0 \int_{\Omega_1} \varphi_{1,0} \psi_1,
$$
and hence $\Lambda_0=\sigma_1^{\Omega_1}[ -\Delta+c_1;\mathcal{N}]$.

In the case in which $\sigma_1^{\Omega_1}[ -\Delta+c_1;\mathcal{N}]=\sigma_1^{\Omega_2}[ -\Delta+c_2;\mathcal{N}]$, assuming that $\Lambda_0< \sigma_1^{\Omega_1}[ -\Delta+c_1;\mathcal{N}]=\sigma_1^{\Omega_2}[ -\Delta+c_2;\mathcal{N}]$ we arrive at a contradiction  with the fact that either $\varphi_{1,0}$ or $\varphi_{2,0}$ are different from zero.
\end{proof}

\noindent{\bf Limit behavior when $\gamma\to \infty$:} On the other hand, the behavior of the eigenvalue when $\gamma\to\infty$ has been studied in~\cite{WangSu} using the variational definition of the principal eigenvalue. We include the result here, providing a different proof, for the reader's convenience.

\begin{theorem}
\label{thm:symmetric.infty}
Let $\Lambda_1(\gamma)$ be the principal eigenvalue of~\eqref{eq:system.symmetric} and ${\bf \Phi}_\g$ the associated positive eigenfunction.
Then,
 $$
\Lambda_\infty:=\lim_{\gamma\to \infty}\Lambda_1(\gamma)=\sigma_1^\Omega[-\Delta+{ c},\mathcal{N}],$$
where
\begin{equation}
\label{defic}
c(x)=\left\{\begin{array}{ll}
c_1(x) & \text{in $\Omega_1$,}\\
c_2(x) & \text{in $\Omega_2$.}\\
\end{array}\right.
\end{equation}
Furthermore, if  ${\bf \Phi}_\g$ is normalized such that $\|{\bf \Phi}_\gamma\|_{\mathcal{L}^2}=1$, then
$$
        \lim_{\gamma\to \infty}{\bf \Phi_\gamma}={\bf \Phi}_\infty,\quad \text{in $H^1(\Omega)$}
$$
where ${\bf \Phi}_\infty\in H^1(\Omega)$ is the eigenfunction associated with $\Lambda_\infty$.
\end{theorem}
\begin{proof}
 Unlike what happens for the asymmetric case, $\Lambda(\gamma)$ is non-decreasing and bounded, so that it
implies that the limit $\Lambda_\infty$ exists. Moreover, thanks to Lemma\;\ref{lemma.convergence.phi} we have there exists also a limit function ${\bf \Phi}_\infty$, such that $\varphi_{i,\infty}\in H^1(\Omega_i)$ with  $\|{\bf \Phi}_\infty\|_{L^2(\Omega)}=1$ so that
\begin{equation}
    \varphi_{i,\gamma}\rightharpoonup \varphi_{i,\infty}\quad  \text{in $H^1(\O_i)$,}\quad
\varphi_{i,\gamma}\to \varphi_{i,\infty} \quad\text{in $L^2(\O_i)$ and $L^2(\Sigma)$}\quad\text{and}\quad \varphi_{1,\gamma}-\varphi_{2,\gamma}\to 0\quad\text{in  $L^2(\Sigma)$.}
\end{equation}
Hence, ${\bf \Phi}_\infty\in H^1(\Omega)$.
Take now $\Psi\in H^1(\Omega)$. Thus, passing to the limit, it yields
$$
\begin{aligned}
    \int_{\Omega_1} \nabla \varphi_{1,\infty}\cdot \nabla \psi_1 + \int_{\Omega_2} \nabla \varphi_{2,\infty}\cdot \nabla \psi_2 &+
    \int_{\Omega_1} c_1(x) \varphi_{1,\infty} \psi_1
    \int_{\Omega_2} c_2(x) \varphi_{2,\infty} \psi_2 \\
    &=\Lambda_\infty \left(\int_{\Omega_1}\varphi_{1,\infty} \psi_1 +\int_{\Omega_2} \varphi_{1,\infty} \psi_2\right).
\end{aligned}
    $$
where $\psi_i=\Psi_{\rfloor\Omega_i}$, $i=1,2$. This  implies that ${\bf \Phi}_\infty$ is a positive  eigenfunction of
$$
     \left\{\begin{array}
            {ll}
            -\Delta {\bf \Phi}+c(x){\bf \Phi}=\lambda {\bf \Phi},\qquad&\text{in\ }\Omega,\\
            \partial_{\nu} {\bf\Phi}=0,&\text{on\ }\Gamma,
        \end{array}\right.
$$
Therefore, by definition, $\Lambda_\infty:=\sigma_1^{\Omega}[-\Delta+c; \mathcal{N}]$.
\end{proof}


\end{document}